\pdfoutput=1
\documentclass[10pt]{article}

\usepackage[a4paper,margin=1.05in]{geometry}
\usepackage[T1]{fontenc}
\usepackage[utf8]{inputenc}
\usepackage{lmodern}
\usepackage{microtype}
\usepackage{xurl}
\usepackage{amsmath,amssymb,amsfonts,amsthm}
\usepackage{mathtools}
\usepackage{enumitem}
\usepackage{tabularx}
\usepackage{graphicx}
\usepackage[sort&compress,numbers,square]{natbib}
\setcitestyle{numbers,square}
\usepackage[colorlinks=true,linkcolor=blue,citecolor=blue,urlcolor=blue]{hyperref}

\graphicspath{{fig/}}

\newcommand{\TV}{\mathrm{TV}}

\newcommand{\Stsp}{\ensuremath{\mathcal{X}}}
\newcommand{\Exsp}{\Stsp_{\rm exc}}
\newcommand{\Rgsp}{\Stsp_{\rm reg}}
\newcommand{\Zsp}{\ensuremath{\mathcal{Z}}}

\newcommand{\sgn}{\mathop{\mathrm{sgn}}}

\newcommand{\1}{\mathbb{1}}

\theoremstyle{plain}
\newtheorem{theorem}{Theorem}[section]
\newtheorem{conjecture}[theorem]{Conjecture}
\newtheorem{lemma}[theorem]{Lemma}
\newtheorem{proposition}[theorem]{Proposition}

\newtheorem{assumption}[theorem]{Assumption}
\theoremstyle{definition}
\newtheorem{definition}{Definition}

\theoremstyle{remark}
\newtheorem{remark}{Remark}

\hypersetup{
  pdftitle={Threshold-indexability of restless bandits \\ with real interval state spaces: a performance-metric verification framework and long-run average analysis},
  pdfauthor={Jose Nino-Mora},
  pdfsubject={Restless bandits; Whittle index; threshold-indexability},
  pdfkeywords={restless bandits, Whittle index, indexability, threshold policies, real-valued state spaces, long-run average criterion, partial conservation laws}
}

\title{Threshold-indexability of restless bandits \\ with real interval state spaces:
 a performance-metric verification framework and long-run average analysis}
\author{Jos\'e Ni\~no-Mora\\[0.25em]
\small Departamento de Estad\'istica, Universidad Carlos III de Madrid\\
\small C.\ Madrid, 126, Getafe, 28903, Madrid, Spain\\
\small \href{mailto:jose.nino@uc3m.es}{jose.nino@uc3m.es}\\
\small \url{https://alum.mit.edu/www/jnimora}\\
\small ORCID: 0000-0002-2172-3983}
\date{}

\begin{document}

\maketitle

\begin{abstract}
Restless multiarmed bandits are Markov decision process models for allocating a
scarce resource among projects whose states evolve under active or passive
actions. Whittle's index policy is widely used for such problems, but its
application to a given model requires both a proof of indexability and a means
of computing the index, two analytically challenging tasks. This paper develops
a performance-metric framework for proving threshold-indexability and computing
Whittle indices for binary-action projects with real interval state spaces. The
framework extends discounted partial conservation law \textup{(PCL)} methods to
a criterion-agnostic setting and works directly with reward and resource metrics
of threshold policies, rather than first proving threshold optimality and then
monotonicity of optimal thresholds in the resource price. The main theorem is a
verification and characterization result: under marginal-resource positivity and
a marginal integration-by-parts identity, threshold-indexability is equivalent
to monotonicity and continuity of the marginal productivity (MP) index,
which then equals the Whittle index. The framework is specialized to the
discrete-time long-run average criterion by a vanishing-discount transfer of
discounted threshold metrics and includes exceptional states where the MP
marginal-resource denominator vanishes, handled by continuous extension or
vanishing-discount limits. Applications to web crawling and noisy-channel
transmission recover known long-run average Whittle indices. For scalar
Kalman-filter bandits, it proves a regular-part average-cost result and reduces
the remaining indexability question to explicit exceptional-state metric-limit
conjectures.
\end{abstract}


\bigskip

\section{Introduction}
\label{s:intro}

\subsection{Restless bandits and Whittle indexability}
\label{s:intro_multiproject}

Restless multiarmed bandits were introduced by \citet{whit88} as Markov
decision process \textup{(MDP)} models for allocating a scarce resource among
several stochastically evolving projects; see \citep{nmmath23} for a recent
review. In the binary-action case considered here, project \(n\) has state
\(X_{n,t}\), action \(A_{n,t}\in\{0,1\}\), with \(A_{n,t}=1\) denoting
activation, and evolves independently of the other projects conditional on its
own state--action history. The projects are coupled only through the resource
constraint, which makes exact dynamic programming high-dimensional.

Whittle's Lagrangian relaxation replaces the hard resource constraint by an
average one. If \(F_n(x_n,\pi_n)\) and \(G_n(x_n,\pi_n)\) denote the reward and
resource metrics of project \(n\), the relaxed problem maximizes total reward
subject to a total resource budget. Introducing a resource price \(\lambda\)
decomposes it into single-project objectives of the form
\(F_n(x_n,\pi_n)-\lambda G_n(x_n,\pi_n)\); for the corresponding discounted
multiproject derivation, see \cite[Online Companion, App.~A]{nmmor20}. In the critical-price
formulation used here, the mathematical core is \emph{indexability}: for each state \(y_n\), there is a unique critical price
\(\lambda_n^*(y_n)\) at which activation and passivity are both optimal; below
that price activation is optimal, and above it passivity is optimal. This
critical price is the Whittle index, and the resulting index policy activates
projects with the largest current indices, subject to the resource constraint.

\subsection{Real interval state spaces and threshold policies}
\label{s:intro_scalar_thresholds}

This paper studies projects whose state space is a real interval
\(\Stsp\subseteq\mathbb R\). Such scalar states arise in sensor scheduling,
remote estimation, opportunistic communication, web crawling, and service
systems. Examples include Kalman-filter scheduling and sensor management
\citep{lascalaMoran06,leNyetal11,danceSi15}, dynamic multichannel access
\citep{liuZhao10}, crawling of ephemeral content \citep{avraBorkar18}, and
remote estimation of continuous Gauss--Markov processes \citep{orneeSun23}.

The scalar order makes threshold policies natural: one action is used below a
threshold and the other above it, typically with activation in high states. Yet
thresholdability alone is not indexability. The prevailing proof route first
establishes threshold optimality for each fixed price \(\lambda\), usually via
a model-specific DP argument, and then proves that the optimal threshold varies
monotonically with \(\lambda\). This second step is often as delicate as the
first, especially on continuous state spaces, where optimality equations are
functional equations and boundary, unboundedness, and long-run average issues can
interact. Thus several real-state analyses either rely on model-specific
monotonicity arguments or prove indexability conditional on thresholdability;
see, for example,
\cite{liuZhao10,leNyetal11,avraBorkar18,danceSi15,fryerHarms18,chenetal22,kriouileAssaad21,orneeSun23}.

\subsection{A performance-metric PCL route}
\label{s:intro_pcl_route}

The partial conservation laws \textup{(PCLs)} approach avoids this two-step DP
route by working directly with threshold-policy performance metrics. It was
introduced for finite-state projects in \cite{nmaap01,nmmp02}, extended to
countable-state projects in \cite{nmmor06}, and later developed for discounted
real-state projects in \cite{nmmor20}. PCL identities relate neighboring
threshold policies, yield marginal reward and marginal resource metrics, and
identify their ratio as a marginal productivity \textup{(MP)} index. Under the
PCL-indexability conditions, that MP index is the Whittle index.

The present paper extends this approach in two directions. First, it formulates a
criterion-agnostic performance-metric framework for binary-action projects on
real intervals. The framework defines reward and resource metrics, Q-value and
marginal metrics, thresholdability, indexability, threshold-indexability, and
the MP index without committing to discounted or average rewards. Second, it
proves an abstract PCL verification and characterization theorem. Once
marginal-resource positivity \textup{(PCLI1)} and the marginal
integration-by-parts identity \textup{(PCLI3b)} hold, threshold-indexability is
equivalent to the MP-index monotonicity/continuity condition
\textup{(PCLI2)}. When these conditions hold, the Whittle index is the MP index
and the optimal threshold maps are precisely the generalized inverses of that
index. The theorem also treats exceptional states at which the diagonal
marginal-resource denominator vanishes; there the index is supplied by a
continuous extension or by a vanishing-discount limit.

\subsection{Long-run average extension}
\label{s:intro_avg_ext}

The main specialization is the discrete-time long-run average criterion. The
finite- and countable-state PCL theories already include average criteria
\citep{nmaap01,nmmp02,nmmor06}; the corresponding real-interval counterpart
remained unavailable. This extension is not a change of notation. Discounted threshold
metrics usually diverge at rate \((1-\beta)^{-1}\) as \(\beta\uparrow1\), so one
must separate average gains from bias terms and then transfer marginal metrics,
MP indices, and PCL identities. Moreover, in average limits the marginal
resource metric \(z\mapsto g_x(z)\) need not have locally bounded variation, so
one cannot generally use a Lebesgue--Stieltjes \textup{(LS)} integral with integrator \(\mathrm dg_x\). The
paper therefore separates the Radon--Nikodym threshold identity
\textup{(PCLI3a)}, which gives the PCL and shadow-price interpretation, from the
marginal integration-by-parts identity \textup{(PCLI3b)}, which is the sign
identity used in the threshold-indexability proof.

\subsection{Applications and scope}
\label{s:contribs}

The framework is applied to three scalar-state models. For optimal web crawling
it recovers the long-run average Whittle index of \cite{avraBorkar18} from the
discounted formulas in \cite[\S~12.1]{nmmor20}. For dynamic transmission over a
noisy channel it recovers the long-run average index of \cite{liuZhao10} from
\cite[\S~12.2]{nmmor20}. In both cases the proof is a metric-transfer argument:
discounted threshold metrics are passed to average reward, resource, bias, and
marginal metrics.

The scalar Kalman-filter bandit is treated differently. Its restless bandit
formulation goes back to \cite{lascalaMoran06}, the PCL approach to the
discrete-time scalar model was introduced in \cite{nmsv09}, and the discounted
PCL-indexability analysis was completed by \citet{danceSi19}. The present
paper proves the regular periodic-itinerary part of the average transfer and
identifies the remaining exceptional-state metric-limit conjectures. Thus the
remaining average-cost indexability problem is reduced to explicit questions aligned
with \textup{(PCLI1)}--\textup{(PCLI3)}: off-diagonal marginal-resource
positivity, continuous extension of the MP index at exceptional states, and the
corresponding singular PCL identities.

\subsection{Average-criterion MDP tools and paper organization}
\label{s:intro_lit}
\label{s:org}

For each fixed price \(\lambda\), the average-reward analysis uses standard
average-cost MDP tools: existence of gain and bias functions, average-cost
optimality equations or inequalities, stationary average-optimal policies, and
vanishing-discount convergence; see, for example, \cite[Ch.~5]{herlerLass96}
and \cite{schal93,jaskiNowak06,costaDufour12,vegaAmaya15,vegaAmaya18,feinbergetal12,feinbergLiang22}. Related structural and convex-analytic approaches
include \cite{yu22,arapostYuksel23}. For noncompact unbounded-cost models,
such as the Kalman-filter example, these DP prerequisites are separate from
the PCL metric-limit analysis.

The rest of the paper is organized as follows. Section~\ref{s:arf} develops
the abstract performance-metric framework. Section~\ref{s:pcli} states the PCL
conditions and proves the verification/characterization theorem.
Section~\ref{s:dtla} gives the vanishing-discount transfer to the long-run
average criterion. Section~\ref{s:aexamples} applies the framework to web
crawling, noisy-channel transmission, and the scalar Kalman-filter bandit.


\section{Performance-metric framework for threshold-indexability on interval state spaces}
\label{s:arf}

\subsection{Restless bandit project, performance metrics, and
\texorpdfstring{$\lambda$}{lambda}-price problems}
\label{s:mdf}

We consider a general single-project control model with a binary action and a
single resource. The project state at decision time \(t \in \mathbb{T}\) is denoted by
\(X_t\in\Stsp\), where the \emph{state space} \(\Stsp\) is a closed real
interval bounded below,
\(
\Stsp=\{x\in\mathbb R:\ell\le x\le u\},
\)
where \(\ell\in\mathbb R\), \(\ell<u\), and either \(u\in\mathbb R\) or
\(u=+\infty\). Thus \(\Stsp\) is either a compact interval or the half-line
\([\ell,\infty)\). At each decision epoch, the controller chooses a binary
action \(A_t\in\{0,1\}\), where action \(a=1\) denotes activation and action
\(a=0\) denotes passivity.

We leave open the time parameterization of decision epochs, except that time
starts at \(0\) and the horizon is infinite. Thus, decision opportunities may
arise in discrete time, at periods \(t \in \mathbb{T} = \{0,1,\ldots\}\), or in continuous time, for
example at intervention epochs. Likewise, the probabilistic dynamics are left
unspecified in this general framework.

Let \(\Pi\) denote the class of admissible project-operating policies. For each
initial state \(x\in\Stsp\) and policy \(\pi\in\Pi\), let \(F(x,\pi)\) and
\(G(x,\pi)\) denote the corresponding \emph{reward} and
\emph{resource-consumption performance metrics}. These are finite real-valued
performance quantities, typically defined as expectations or expected limits
over the state-action path \(\{(X_t,A_t)\colon t \in \mathbb{T}\}\), conditional on
\(X_0=x\). They measure, respectively, the reward earned and the resource
expended under policy \(\pi\) starting from \(x\). When the resource is time,
we refer to \(G(x,\pi)\) as the \emph{work metric}. The definitions of \(F\)
and \(G\) are model- and criterion-dependent; in particular, they may be
discounted metrics or long-run average metrics.

When resource usage is charged at price \(\lambda\) per unit, the
project's \emph{net value metric} is
\(
V_\lambda(x,\pi)
\triangleq
F(x,\pi)-\lambda G(x,\pi).
\)
The associated \emph{\(\lambda\)-price problem} is to find a policy maximizing
this net value from every initial state:
\begin{equation}
\label{eq:lpricepfg1}
P_\lambda\colon\qquad
\textup{find } \pi_\lambda^*\in\Pi \textup{ such that }
V_\lambda(x,\pi_\lambda^*)=V_\lambda^*(x)
\triangleq \sup_{\pi\in\Pi}V_\lambda(x,\pi),
\qquad x\in\Stsp .
\end{equation}
A policy solving \eqref{eq:lpricepfg1} is called
\emph{\(P_\lambda\)-optimal}, and \(V_\lambda^*(\cdot)\) is called the
project's \emph{optimal net value function}. We consider the collection of all
\(\lambda\)-price problems,
\begin{equation}
\label{eq:Pcollec}
\mathcal P
\triangleq
\{P_\lambda\colon\lambda\in\mathbb R\}.
\end{equation}

We also use one-step action--policy concatenations. For
\(a\in\{0,1\}\) and \(\pi\in\Pi\), let \(\langle a,\pi\rangle\) denote the
policy that chooses action \(a\) at time \(t=0\) and then follows policy
\(\pi\) thereafter. We assume that
\(
\langle a,\pi\rangle\in\Pi .
\)

To express one-step optimality conditions, we use Q-value notation. For each
admissible policy \(\pi\), state \(x\), and action \(a\), let
\(Q_F^\pi(x,a)\) and \(Q_G^\pi(x,a)\) denote the one-step reward and resource
\emph{Q-value metrics} associated with the criterion under consideration. These
Q-value metrics are part of the performance-metric specification of the model.
In discounted problems, they can be taken to be the literal concatenation
metrics
\(Q_F^\pi(x,a)=F(x,\langle a,\pi\rangle)\) and
\(Q_G^\pi(x,a)=G(x,\langle a,\pi\rangle)\).
In long-run average problems, where \(F\) and \(G\) may denote average reward
and resource-consumption rates, a one-time initial action generally does not
change such rates. In that case \(Q_F^\pi\) and \(Q_G^\pi\) are instead the
corresponding differential, or \emph{bias}, Q-value metrics, as specified in
the model-specific average-reward formulation.

The \emph{net Q-value metric} at price \(\lambda\) is
\(
Q_\lambda^\pi(x,a)
\triangleq
Q_F^\pi(x,a)-\lambda Q_G^\pi(x,a).
\)
Associated with the Q-value metrics, we define the \emph{marginal reward} and
\emph{marginal resource} metrics by
\begin{equation}
\label{eq:def-fg-marginal}
f(x,\pi)
\triangleq
Q_F^\pi(x,1)-Q_F^\pi(x,0),
\qquad
g(x,\pi)
\triangleq
Q_G^\pi(x,1)-Q_G^\pi(x,0).
\end{equation}
Thus \(f(x,\pi)\) and \(g(x,\pi)\) measure the marginal changes in reward and
resource consumption obtained by choosing the active action initially rather
than the passive action, while following policy \(\pi\) thereafter, in the
Q-value sense appropriate to the chosen performance criterion.

For each price \(\lambda\), the associated \emph{marginal net value metric} is
\begin{equation}
\label{eq:def-v-marginal}
v_\lambda(x,\pi)
\triangleq
f(x,\pi)-\lambda g(x,\pi)
=
Q_\lambda^\pi(x,1)-Q_\lambda^\pi(x,0).
\end{equation}

We assume that each \(\lambda\)-price problem admits an optimal
\emph{stationary deterministic} \textup{(SD)} policy in \(\Pi\). For an SD
policy \(\pi\), we write \(\pi(x)\in\{0,1\}\) for the action prescribed in
state \(x\). We further assume that \(P_\lambda\)-optimal actions are
characterized by Q-value inequalities, as stated next. Throughout the paper,
``iff'' means ``if and only if''.

\begin{assumption}[Existence of optimal SD policies and optimal action characterization]
\label{ass:excharcosdp}
For each price \(\lambda\in\mathbb R\):
\begin{enumerate}[label=\textup{(\roman*)},leftmargin=*]
\item
There exists an SD \(P_\lambda\)-optimal policy in \(\Pi\).

\item
There exist optimal net Q-value functions
\(
Q_\lambda^*(\cdot,a)\colon\Stsp\to\mathbb R\) for
\(a\in\{0,1\},\)
such that the sign of
\(
v_\lambda^*(x)
\triangleq
Q_\lambda^*(x,1)-Q_\lambda^*(x,0)
\)
characterizes \(P_\lambda\)-optimal actions:
in state \(x\), action \(a=1\) is optimal iff \(v_\lambda^*(x)\ge0\), and
action \(a=0\) is optimal iff \(v_\lambda^*(x)\le0\). Moreover, any admissible
SD policy that selects, in every state, an action characterized as
\(P_\lambda\)-optimal by these inequalities is \(P_\lambda\)-optimal.

\item
An admissible SD policy \(\pi^*\) is \(P_\lambda\)-optimal iff, for every
\(x\in\Stsp\),
\[
v_\lambda(x,\pi^*)\ge0
\quad\textup{whenever}\quad
\pi^*(x)=1,
\qquad
v_\lambda(x,\pi^*)\le0
\quad\textup{whenever}\quad
\pi^*(x)=0.
\]
In that case, the functions \(Q_\lambda^*(\cdot,a)\) in \textup{(ii)} may be
chosen as
\(Q_\lambda^*(x,a) \triangleq Q_\lambda^{\pi^*}(x,a),\)
and hence
\(v_\lambda^*(x)=v_\lambda(x,\pi^*).\)
\end{enumerate}
\end{assumption}

\begin{remark}
\label{re:excharcosdp}
\begin{enumerate}[label=\textup{(\roman*)},leftmargin=*]
\item The role of Assumption~\ref{ass:excharcosdp} is to ensure that optimality in
each \(\lambda\)-price problem can be checked through the signs of the
corresponding marginal net Q-value metrics. In discounted models it is typically verified
from the discounted optimality equation. In average-reward models it is
verified from the appropriate optimality equation or inequality,
with \(Q_\lambda^\pi\) interpreted as a bias Q-value metric.
\item By Assumption~\ref{ass:excharcosdp}\textup{(ii)}, both actions are
\(P_\lambda\)-optimal in state \(x\) iff
\begin{equation}
\label{eq:bothactopt}
v_\lambda^*(x)=0.
\end{equation}
Equivalently, under an SD \(P_\lambda\)-optimal continuation policy
\(\pi^*\), Assumption~\ref{ass:excharcosdp}\textup{(iii)} allows choosing
\[
Q_\lambda^*(x,a)
=
Q_\lambda^{\pi^*}(x,a),
\qquad a\in\{0,1\},
\]
so that both initial actions have the same net Q-value:
\(
Q_\lambda^{\pi^*}(x,0)
=
Q_\lambda^{\pi^*}(x,1).
\)
For fixed \(x\), any solution of \eqref{eq:bothactopt}, viewed as an
\emph{indifference equation} in \(\lambda\), is precisely a price
at which both actions are optimal in state \(x\).
\end{enumerate}
\end{remark}

\subsection{Indexability}
\label{s:indxb}

We now formalize \emph{indexability} \citep{whit88} in this framework.
It describes how the optimal actions in each fixed state vary with
the resource price. Once the index map is known, the optimal action in state
\(x\) at price \(\lambda\) is obtained by comparing \(\lambda\) with the index
value at \(x\). Thus indexability gives an explicit solution of the whole price problem
collection \(\mathcal P\).

For a given price \(\lambda\), define the \(P_\lambda\)-optimal active and
passive regions by
\begin{equation}
\label{eq:Sstarlam}
S_\lambda^{*,1}
\triangleq
\{x\in\Stsp\colon v_\lambda^*(x)\ge0\},
\qquad
S_\lambda^{*,0}
\triangleq
\{x\in\Stsp\colon v_\lambda^*(x)\le0\}.
\end{equation}
These regions need not be disjoint. Their intersection is the set of states at
which both actions are \(P_\lambda\)-optimal.

In the sequel, \(\sgn(\cdot)\) denotes the sign function
\(\sgn\colon\mathbb R\to\{-1,0,1\}\), defined by
\[
\sgn(y)
\triangleq
\begin{cases}
-1, & y<0,\\
0, & y=0,\\
1, & y>0.
\end{cases}
\]

We use the critical-price formulation of indexability in \cite{nmmor20}, where each state has a
finite unique indifference price.

\begin{definition}[Indexability]
\label{def:indx}
The project is \emph{indexable} if there exists a map
\(\lambda^*\colon\Stsp\to\mathbb R\), called  its \emph{Whittle index}, such that
\begin{equation}
\label{eq:indxsigndef}
\sgn\bigl(v_\lambda^*(x)\bigr)
=
\sgn\bigl(\lambda^*(x)-\lambda\bigr),
\qquad x\in\Stsp,\ \lambda\in\mathbb R.
\end{equation}
\end{definition}

\begin{remark}
\label{re:indx}
\begin{enumerate}[label=\textup{(\roman*)},leftmargin=*]
\item
Definition~\ref{def:indx} implies the following price monotonicity at each fixed
state \(x\): action \(a=1\) is uniquely optimal iff
\(\lambda<\lambda^*(x)\), action \(a=0\) is uniquely optimal iff
\(\lambda>\lambda^*(x)\), and both actions are optimal iff
\(\lambda=\lambda^*(x)\). Thus the optimal action is nonincreasing in the
resource price.

\item
Consequently, for each fixed state \(x\), the indifference equation
\eqref{eq:bothactopt}, viewed as an equation in \(\lambda\), has the unique
solution \(\lambda=\lambda^*(x)\). Thus the Whittle index is the unique
resource price at which the controller is indifferent between the active and
passive actions in state \(x\).

\item
Definition~\ref{def:indx} does not impose monotonicity of optimal actions with
respect to the state \(x\). State monotonicity, and hence optimal threshold
structure, is a separate property considered below.

\item
For an indexable project,
\[
S_\lambda^{*,1}
=
\{x\in\Stsp\colon\lambda^*(x)\ge\lambda\},
\qquad
S_\lambda^{*,0}
=
\{x\in\Stsp\colon \lambda^*(x)\le\lambda\}.
\]
Hence, as \(\lambda\) increases from \(-\infty\) to \(+\infty\), the optimal
active region contracts from \(\Stsp\) to \(\varnothing\), while the optimal
passive region expands from \(\varnothing\) to \(\Stsp\).

\item
Definition~\ref{def:indx} is the sign-function formulation of indexability
considered in \cite{nmmor20} for the discounted criterion. It corresponds to
Whittle's \cite{whit88} monotone-passive-region formulation in the common case where, for
each state, the active--passive indifference price is unique. In the present
formulation, increasing \(\lambda\) penalizes resource usage and therefore
makes the passive action more attractive.

\item
The index map satisfying \eqref{eq:indxsigndef}, if it exists, is unique.
Indeed, if two finite values \(\lambda_1^*(x)<\lambda_2^*(x)\) gave the same
sign characterization at state \(x\), then choosing \(\lambda\) strictly
between them would give different signs.
\end{enumerate}
\end{remark}

\subsection{Thresholdability and threshold-indexability}
\label{s:assub}

Following \cite{nmmor20}, we now focus on projects whose optimal actions 
have a threshold structure in the state variable. In such cases, for each price
\(\lambda\), optimal actions can be characterized by a threshold separating a
lower passive region from an upper active region.
We thus fix the convention that activation is associated with high states; models
with the reverse ordering can be handled by reversing the state order.

We use two threshold-policy conventions, depending on whether the active action
is taken at the threshold. For \(z\in\mathbb R\cup\{\infty\}\), the passive-at-threshold
\emph{\(z\)-policy} has active and passive regions
\[
S^1(z)
\triangleq
\{x\in\Stsp\colon x>z\},
\qquad
S^0(z)
\triangleq
\{x\in\Stsp\colon x\le z\}.
\]
If \(z=\infty\), these inequalities are interpreted in the extended-real sense,
so that
\(S^1(\infty)=\varnothing\) and \(S^0(\infty)=\Stsp\).

For finite \(z\), the active-at-threshold \emph{\(z^{\scriptscriptstyle -}\)-policy} has active and passive regions
\[
S^1(z^{\scriptscriptstyle -})
\triangleq
\{x\in\Stsp\colon x\ge z\},
\qquad
S^0(z^{\scriptscriptstyle -})
\triangleq
\{x\in\Stsp\colon x<z\}.
\]
The \(z\)- and \(z^{\scriptscriptstyle -}\)-policies differ only at the
threshold state \(x=z\), and only when \(z\in\Stsp\).

We write \(\Zsp\) for the \emph{admissible threshold set}. We assume throughout
that all finite-threshold policies introduced above are admissible. The only
threshold policy whose admissibility may fail in the unbounded state case is
the all-passive policy, represented by \(z=\infty\). For example, under the
long-run average criterion on an unbounded state space, the all-passive policy
may be unstable or may fail to yield finite, well-defined reward and resource
metrics. Thus
\[
\Zsp
\triangleq
\begin{cases}
\mathbb R\cup\{\infty\},
& \textup{if } u=\infty \textup{ and the all-passive policy is admissible},\\[1mm]
\mathbb R,
& \textup{otherwise}.
\end{cases}
\]

The all-active policy is represented by any threshold \(z<\ell\) under the
\(z\)-policy, and by any threshold \(z\le\ell\) under the
\(z^{\scriptscriptstyle -}\)-policy. When \(u<\infty\), the all-passive policy
is represented by any threshold \(z\ge u\) under the \(z\)-policy, and by any
threshold \(z>u\) under the \(z^{\scriptscriptstyle -}\)-policy. When
\(u=\infty\), no finite threshold represents the all-passive policy; in that
case it is represented by \(z=\infty\), provided that \(\infty\in\Zsp\).

For \(z\in\Zsp\), we write \(F(x,z)\) and \(G(x,z)\) for the reward and
resource metrics of the \(z\)-policy. For finite \(z\), we write
\(F(x,z^{\scriptscriptstyle -})\) and \(G(x,z^{\scriptscriptstyle -})\) for
the corresponding metrics of the \(z^{\scriptscriptstyle -}\)-policy. The same
substitution convention applies to Q-value, marginal, and net metrics.

For convenience, we use the initial state as a subscript for threshold-policy
metrics:
\[
        F_x(z)\triangleq F(x,z),
        \qquad
        G_x(z)\triangleq G(x,z),
\]
and define the threshold-policy net metric
\[
        V_{\lambda,x}(z)
        \triangleq
        V_\lambda(x,z)
        =
        F_x(z)-\lambda G_x(z).
\]
Similarly, for marginal metrics,
\[
        f_x(z)\triangleq f(x,z),
        \qquad
        g_x(z)\triangleq g(x,z),
        \qquad
        v_{\lambda,x}(z)
        \triangleq
        v_\lambda(x,z)
        =
        f_x(z)-\lambda g_x(z).
\]
For finite \(z\), the same notation with \(z^{\scriptscriptstyle -}\) denotes
the corresponding quantities under the active-at-threshold policy. Under the
regularity assumptions imposed below, the notation with
\(z^{\scriptscriptstyle -}\) agrees with the ordinary left limit in the
threshold parameter.

We now formulate thresholdability. The definition requires that, for each
price \(\lambda\), there be a threshold \(z^*(\lambda)\in\Zsp\) such that the
passive action is optimal throughout the lower region and the active action is
optimal throughout the upper region. If the threshold itself is a state, both
actions are optimal there. The definition does not rule out additional
indifference states away from the selected threshold; this flexibility is
important because the MP index need only be nondecreasing, and may have flat
portions.

\begin{definition}[Thresholdability]
\label{def:potp}
The project is \emph{thresholdable} if there exists a map
\(z^*\colon\mathbb R\to\Zsp\) such that, for every \(x\in\Stsp\) and
\(\lambda\in\mathbb R\),
\begin{equation}
\label{eq:thrlsigndef}
\begin{split}
x\le z^*(\lambda)
&\quad\Longrightarrow\quad
v_\lambda^*(x)\le0,\\
x\ge z^*(\lambda)
&\quad\Longrightarrow\quad
v_\lambda^*(x)\ge0.
\end{split}
\end{equation}
If \(z^*(\lambda)=\infty\), the inequalities are interpreted in the
extended-real order; in particular, the second implication is then vacuous. We
call any such \(z^*(\cdot)\) an \emph{optimal threshold map}.
\end{definition}

\begin{remark}[Equivalent characterization of thresholdability]
\label{re:thr_equiv}
\begin{enumerate}[label=\textup{(\roman*)},leftmargin=*]
\item
If the project is thresholdable and \(z^*(\cdot)\) is an optimal threshold map,
then, for every \(\lambda\in\mathbb R\), the \(z^*(\lambda)\)-policy is
\(P_\lambda\)-optimal and
\[
V_\lambda^*(x)
=
V_{\lambda,x}(z^*(\lambda)),
\qquad x\in\Stsp.
\]
Moreover, if \(z^*(\lambda)\in\mathbb R\), then the
\(z^*(\lambda)^{\scriptscriptstyle -}\)-policy is also \(P_\lambda\)-optimal
and
\[
V_\lambda^*(x)
=
V_{\lambda,x}(z^*(\lambda)^{\scriptscriptstyle -}),
\qquad x\in\Stsp.
\]

\item
Equivalently, the project is thresholdable iff, for every
\(\lambda\in\mathbb R\), there exists \(z\in\Zsp\) such that the \(z\)-policy is
\(P_\lambda\)-optimal and, whenever \(z\in\mathbb R\), the
\(z^{\scriptscriptstyle -}\)-policy is also \(P_\lambda\)-optimal.

\item
Thresholdability does not imply that the optimal active and passive regions
\(S_\lambda^{*,1}\) and \(S_\lambda^{*,0}\) are intervals. It only ensures the
existence of a threshold at which the lower region can be made passive and the
upper region active optimally. Additional indifference states may still occur
away from the threshold.
\end{enumerate}
\end{remark}

The two properties just defined address different aspects of the
single-project price problems. Indexability is a price-comparative-statics
property: at each state, there is a critical resource price that determines
whether activation or passivity is optimal. Thresholdability is a
state-monotonicity property: for each price, at least one optimal stationary
policy has threshold form. The central property herein is
their conjunction. In that case, the whole family of price problems is solved
by threshold policies, and the thresholds are ordered consistently by the
Whittle index.

\begin{definition}[Threshold-indexability]
\label{def:indxtp}
The project is \emph{threshold-indexable} if it is thresholdable and indexable.
\end{definition}

\subsection{The MP metric and the MP index}
\label{s:mrrum}

We use the threshold-policy notation introduced above: \(F_x(z),G_x(z)\) for
performance metrics, \(f_x(z),g_x(z)\) for marginal metrics, and
\(V_{\lambda,x}(z)\), \(v_{\lambda,x}(z)\) for their corresponding net
metrics.

Whenever \(g_x(z)>0\), define the \emph{marginal productivity} \textup{(MP)}
\emph{metric} by
\begin{equation}
\label{eq:mpm1}
m_x(z)
\triangleq
\frac{f_x(z)}{g_x(z)}.
\end{equation}
Similarly, for finite \(z\in\mathbb R\), whenever
\(g_x(z^{\scriptscriptstyle -})>0\), define
\(m_x(z^{\scriptscriptstyle -}) \triangleq
f_x(z^{\scriptscriptstyle -})/g_x(z^{\scriptscriptstyle -}).
\)

If \(g_x(x)>0\) for every \(x\in\Stsp\), the project's \emph{MP index} is the
diagonal ratio
\begin{equation}
\label{eq:phimpx}
m(x)
\triangleq
m_x(x)
=
\frac{f_x(x)}{g_x(x)},
\qquad x\in\Stsp .
\end{equation}

\begin{remark}
\label{re:mpm}
The MP metric \(m_x(z)\) is explicitly defined in terms of threshold-policy
marginal metrics, while the MP index \(m(x)\) is its diagonal restriction. This
contrasts with the Whittle index, which is defined implicitly through the
optimality relations \eqref{eq:bothactopt} and \eqref{eq:indxsigndef}. One of
the main messages of the paper is that, under suitable structural conditions,
the MP index and the Whittle index coincide.
\end{remark}

In some settings, the diagonal positivity condition \(g_x(x)>0\) may fail. As
illustrated by the Kalman-filter model under the long-run average criterion in
\S\ref{s:kalmanavg}, it may happen that
\(g_x(x)=0\)
on a subset \(\Exsp\subset\Stsp\) of \emph{exceptional states}, while
\(g_x(x)>0\) holds on the \emph{regular-state set}
\(
\Rgsp
\triangleq
\Stsp\setminus\Exsp .
\)
At exceptional states the diagonal ratio in \eqref{eq:phimpx} is not available.
Nevertheless, the appropriate index value may still be recovered as a
vanishing-discount limit of discounted indices or by continuous extension from
nearby regular states.

This motivates a relaxation of the diagonal positivity requirement $g_x(x) > 0$. 
When exceptional states are present, we call \(m\) an \emph{extended MP index}
if it agrees with the ratio \eqref{eq:phimpx} on
\(\Rgsp\) and is assigned values on \(\Exsp\) by a specified
extension rule, such as continuous extension from \(\Rgsp\) or a
vanishing-discount limit of discounted MP indices.

\section{PCL-indexability conditions and threshold-indexability verification theorem}
\label{s:pcli}

We continue to work under the abstract framework of \S\ref{s:arf}.

\subsection{Lebesgue--Stieltjes conventions and regularity of performance metrics}
\label{s:ls-prelim}

We first fix the analytic conventions used in the PCL conditions. All
functions in this subsection are real-valued and all measures are Borel
measures on \(\mathbb R\). For a right-continuous function
\(h\colon\mathbb R\to\mathbb R\), write
\(h(z-)\triangleq\lim_{y\uparrow z}h(y)\), whenever the limit exists, and
\(\Delta h(z)\triangleq h(z)-h(z-)\). The function \(h\) is
\emph{c\`adl\`ag} if it is right-continuous and has finite left limits.
For an interval \(I\subseteq\mathbb R\), its \emph{total variation} is
\[
\TV_I(h)
\triangleq
\sup\left\{
\sum_{i=1}^{n}|h(z_i)-h(z_{i-1})|:
 z_0<z_1<\cdots<z_n,\ z_i\in I
\right\}.
\]
We say that \(h\) has \emph{bounded variation} on \(I\) if
\(\TV_I(h)<\infty\), and \emph{locally bounded variation} if this holds on every
compact interval.

If \(h\) is nondecreasing and right-continuous, the set function
\begin{equation}
\label{eq:ls-increments}
\mu_h((z_1,z_2])=h(z_2)-h(z_1),
\qquad z_1<z_2,
\end{equation}
extends uniquely to a locally finite nonnegative
\emph{Lebesgue--Stieltjes \textup{(LS)} measure} on
\(\mathcal B(\mathbb R)\). If \(h\) is c\`adl\`ag and has locally bounded
variation, its local Jordan decomposition induces a unique locally finite
signed LS measure \(\mu_h\) satisfying \eqref{eq:ls-increments}; see, e.g.,
\cite[\S X.6]{doob94}. We denote its total variation measure by
\(|\mu_h|\). Thus \(\mu_h(\{z\})=\Delta h(z)\), and, for a measurable
\(\varphi\),
\[
\int_B\varphi(z)\,\mathrm dh(z)
\triangleq
\int_B\varphi(z)\,\mu_h(dz),
\qquad
\int_B|\varphi(z)|\,|\mu_h|(dz)<\infty .
\]
In particular,
\(h(z_2)-h(z_1)=\int_{(z_1,z_2]}1\,\mathrm dh(z)\).

We also use Radon--Nikodym derivatives of signed LS measures. If \(\mu\) is
a locally finite signed measure and \(\nu\) a locally finite nonnegative
measure, \(\mu\ll\nu\) means that \(\nu(B)=0\) implies \(\mu(B)=0\). Since
locally finite Borel measures on \(\mathbb R\) are \(\sigma\)-finite,
\(\mu\ll\nu\) implies the existence of a measurable density
\(\mathrm d\mu/\mathrm d\nu\), unique \(\nu\)-a.e. When \(h\) and \(q\) are
c\`adl\`ag functions of locally bounded variation, the notation
\(\mathrm dh(z)=\varphi(z)\,\mathrm dq(z)\) means the signed-measure
identity
\[
\mu_h(B)=\int_B\varphi(z)\,\mu_q(dz),
\qquad B\in\mathcal B(\mathbb R),
\]
with integrability understood with respect to \(|\mu_q|\). Hence, for
finite \(z_1<z_2\),
\begin{equation}
\label{eq:hz1z2int}
h(z_2)-h(z_1)
=
\int_{(z_1,z_2]}\varphi(z)\,\mathrm dq(z).
\end{equation}

The resource threshold metric \(G_x(z)\) is assumed below to be
nonincreasing in the threshold. Hence \(\mu_{G_x}=-\mu_{-G_x}\), where
\(\mu_{-G_x}\) is the nonnegative LS measure generated by the resource
saved when the threshold is raised. Thus
\(\mathrm dh=\varphi\,\mathrm dG_x\) is equivalently
\(\mathrm dh=-\varphi\,\mathrm d[-G_x]\); at a downward jump of \(G_x\),
\(\mu_{-G_x}(\{z\})=G_x(z-)-G_x(z)\).

Finite-threshold metrics are viewed as functions of \(z\in\mathbb R\), with
constant continuation outside the effective state range. Thus, for
\(H\in\{F,G,f,g\}\), set \(H_x(z)=H_x(\ell^{\scriptscriptstyle -})\) for
\(z<\ell\), and, if \(u<\infty\), set \(H_x(z)=H_x(u)\) for \(z\ge u\). If
\(u=\infty\) and \(\infty\in\Zsp\), values at \(z=\infty\) are interpreted
as one-sided limits of finite-threshold metrics, whenever those limits
exist.

\begin{assumption}[Regularity of threshold performance metrics]
\label{ass:regFG}
For each initial state \(x\in\Stsp\):
\begin{enumerate}[label=\textup{(\roman*)},leftmargin=*]
\item
The threshold performance metrics \(F_x,G_x\colon\mathbb R\to\mathbb R\)
are c\`adl\`ag and have locally bounded variation. Their left limits are the
active-at-threshold metrics:
\(F_x(z^{\scriptscriptstyle -})=F_x(z-)\) and
\(G_x(z^{\scriptscriptstyle -})=G_x(z-)\). If \(\infty\in\Zsp\), then
\(F_x(\infty)=\lim_{z\to\infty}F_x(z)\) and
\(G_x(\infty)=\lim_{z\to\infty}G_x(z)\), with finite limits.

\item
The resource metric \(G_x(z)\) is nonincreasing in \(z\) and is not
identically constant on \(\mathbb R\).
\end{enumerate}
\end{assumption}

\begin{assumption}[Regularity of marginal threshold metrics]
\label{ass:regfg}
For each state \(x\in\Stsp\), the marginal threshold metrics
\(f_x,g_x\colon\mathbb R\to\mathbb R\) are c\`adl\`ag. Their left limits are
the active-at-threshold metrics:
\(f_x(z^{\scriptscriptstyle -})=f_x(z-)\) and
\(g_x(z^{\scriptscriptstyle -})=g_x(z-)\). If \(\infty\in\Zsp\), then
\(f_x(\infty)=\lim_{z\to\infty}f_x(z)\) and
\(g_x(\infty)=\lim_{z\to\infty}g_x(z)\), with finite limits.
\end{assumption}

Under Assumption~\ref{ass:regFG}, the threshold PCL identity will be written
in the signed form
\begin{equation}
\label{eq:pcli3-signed-convention}
F_x(z_2)-F_x(z_1)
=
\int_{(z_1,z_2]}m(z)\,\mathrm dG_x(z),
\qquad z_1<z_2,
\end{equation}
or equivalently \(\mathrm dF_x(z)=m(z)\,\mathrm dG_x(z)\). Because
\(G_x\) is nonincreasing, the same identity has density \(-m\) with respect
to the positive resource-decrement measure \(\mu_{-G_x}\). At an atom,
\eqref{eq:pcli3-signed-convention} gives
\[
F_x(z)-F_x(z-)=m(z)\bigl(G_x(z)-G_x(z-)\bigr).
\]
Therefore, if \(G_x(z-)-G_x(z)>0\),
\begin{equation}
\label{eq:pcl-atom-slope}
m(z)
=
\frac{F_x(z-)-F_x(z)}{G_x(z-)-G_x(z)}
=
\frac{F_x(z^{\scriptscriptstyle -})-F_x(z)}
     {G_x(z^{\scriptscriptstyle -})-G_x(z)}.
\end{equation}
Thus \(m(z)\) is the reward lost per unit of resource saved when the
threshold is raised through \(z\), or, equivalently, the local resource price
that balances the two threshold conventions.

We shall also use the LS integration-by-parts consequence of
\eqref{eq:pcli3-signed-convention}. If \(m\) is continuous and
nondecreasing, then \(\mathrm dm\) is a nonnegative nonatomic LS measure, and
for finite \(z_1<z_2\),
\begin{equation}
\label{eq:nonmarginal-ibp-prelim}
\bigl[F_x(z_2)-m(z_2)G_x(z_2)\bigr]
-
\bigl[F_x(z_1)-m(z_1)G_x(z_1)\bigr]
=
-\int_{(z_1,z_2]}G_x(z)\,\mathrm dm(z),
\end{equation}
by the standard LS integration-by-parts formula
\cite[Theorem~6.2.2]{cartvanBrunt00}. The expanded notation is deliberate:
the coefficient in each bracket is the threshold-dependent value \(m(z)\),
not a fixed price \(\lambda\).

\begin{remark}
\label{re:structural}
The two regularity requirements used below are intentionally asymmetric.
Assumption~\ref{ass:regFG} is needed because the threshold PCL identity
\textup{(PCLI3a)} is an LS identity,
\(\mathrm dF_x=m\,\mathrm dG_x\). In contrast, no bounded-variation
assumption is imposed on \(f_x\) or \(g_x\): in \textup{(PCLI3b)}, \(g_x\)
is only an integrand with respect to \(\mathrm dm\), not an LS integrator.
This distinction is essential in long-run average scalar-state models, where
\(z\mapsto g_x(z)\) may fail to have locally bounded variation.
\end{remark}

\subsection{PCL-indexability conditions}
\label{s:mpmth}

We now state the PCL-indexability conditions. The MP metric is
\(m_x(z)=f_x(z)/g_x(z)\) when \(g_x(z)>0\). The unsubscripted function
\(m\) denotes the candidate MP index: on regular states it is the diagonal
ratio \(m_x(x)\), and at exceptional states it is supplied by continuous
extension. Relative to the discounted framework of \cite{nmmor20}, the third
condition is split into the Radon--Nikodym threshold identity
\textup{(PCLI3a)} and the marginal integration-by-parts identity
\textup{(PCLI3b)}. The former gives the PCL/shadow-price interpretation of
\(m\); the latter is the sign identity used in the verification theorem.

\begin{definition}[PCL-indexability with possible exceptional states]
\label{def:pcli}
The project is \emph{PCL-indexable} with respect to threshold policies if
there exists a Borel set \(\Exsp\subseteq\Stsp\), possibly empty, such that
\textup{(PCLI1)}, \textup{(PCLI2)}, \textup{(PCLI3a)}, and
\textup{(PCLI3b)} below hold.

\begin{enumerate}[
label=\textup{(PCLI\arabic*)},
ref=\textup{(PCLI\arabic*)},
align=left,
labelsep=0.75em,
leftmargin=!
]

\item
The marginal resource metric is positive except possibly on the diagonal over
\(\Exsp\):
\[
        g_x(z)>0,
        \qquad
        x\in\Stsp,\ z\in\Zsp,
        \quad
        \textup{whenever either } z\neq x \textup{ or } x\notin\Exsp .
\]
At exceptional states \(x\in\Exsp\), \(g_x(x)=f_x(x)=0\).

\item
Let \(\Rgsp\triangleq\Stsp\setminus\Exsp\). There exists a nondecreasing
continuous function \(m\) on \(\Stsp\) such that
\[
        m(x)=m_x(x)=\frac{f_x(x)}{g_x(x)},
        \qquad x\in\Rgsp.
\]
For finite-threshold integrals, \(m\) is extended outside the state space by
constant continuation:
\[
        m(z)=m(\ell), \quad z<\ell,
        \qquad
        m(z)=m(u), \quad z\ge u,\ u<\infty .
\]
If \(u=\infty\) and \(\infty\notin\Zsp\), we also require
\(m(x)\to\infty\) as \(x\to\infty\). The resulting finite-threshold
extension induces a nonnegative nonatomic LS measure \(\mu_m\), denoted under
the integral sign by \(\mathrm dm\).

\item
For each initial state \(x\in\Stsp\), the following two identities hold.

\begin{enumerate}[
label=\textup{(PCLI3\alph*)},
ref=\textup{(PCLI3\alph*)},
align=left,
labelsep=0.75em,
leftmargin=*
]

\item
\emph{Radon--Nikodym identity.}
For finite \(z_1<z_2\),
\begin{equation}
\label{eq:pcli3FG}
        F_x(z_2)-F_x(z_1)
        =
        \int_{(z_1,z_2]}m(y)\,\mathrm dG_x(y).
\end{equation}
Equivalently, \(\mathrm dF_x(y)=m(y)\,\mathrm dG_x(y)\) as a signed LS
measure identity.

\item
\emph{Marginal integration-by-parts identity.}
For every finite threshold \(z\in\mathbb R\),
\begin{equation}
\label{eq:pcli3fg-ibp}
        f_x(z)-m(z)g_x(z)
        =
        \begin{cases}
        \displaystyle
        \int_{(z,x)}g_x(y)\,\mathrm dm(y), & z<x,\\[2.5ex]
        \displaystyle
        -\int_{[x,z]}g_x(y)\,\mathrm dm(y), & x\le z .
        \end{cases}
\end{equation}
If \(\infty\in\Zsp\), the second line is also understood at
\(z=\infty\) as the limiting identity
\[
        f_x(\infty)-m(\infty)g_x(\infty)
        =
        -\int_{[x,\infty)}g_x(y)\,\mathrm dm(y),
\]
whenever \(m(\infty)=\lim_{y\to\infty}m(y)<\infty\); the improper integral
is the monotone limit over \([x,b]\) as \(b\to\infty\).

\end{enumerate}
\end{enumerate}
\end{definition}

\begin{remark}
\label{re:pcliii}
Condition \textup{(PCLI3a)} identifies \(m\) as the marginal
reward/resource density along the threshold family. At a threshold atom it is
exactly the chord-slope formula \eqref{eq:pcl-atom-slope}. Condition
\textup{(PCLI3b)} links the marginal metrics to the order of the MP index;
together with \textup{(PCLI1)} and \textup{(PCLI2)}, it yields
\[
\sgn\bigl(f_x(z)-m(z)g_x(z)\bigr)
=
\sgn\bigl(m(x)-m(z)\bigr),
\]
the key statewise comparison used below. When \(\Exsp\neq\varnothing\), the
diagonal ratio \(f_x(x)/g_x(x)\) is unavailable at exceptional states; the
index value there is supplied by the continuous extension in
\textup{(PCLI2)}, often through a vanishing-discount limit. Finally,
\textup{(PCLI1)} is imposed for every admissible threshold, including
\(z=\infty\) when the all-passive policy is admissible; this boundary
positivity is used only in the generalized-inverse boundary cases.
\end{remark}

We next define generalized inverse thresholds. If \(u=\infty\) and
\(\infty\in\Zsp\), set
\(m(\infty)\triangleq\lim_{x\to\infty}m(x)\), with the limit understood in
the extended-real sense; the index remains finite-valued on \(\Stsp\). For
\(\lambda\in\mathbb R\), define
\begin{equation}
\label{eq:tslambm}
\Zsp_\lambda
\triangleq
\begin{cases}
\{z\in\Zsp:m(z)=\lambda\},
&
\textup{if } \{z\in\Zsp:m(z)=\lambda\}\neq\varnothing,
\\[1mm]
\{z\in\Zsp\colon z<\ell\},
&
\textup{if } \lambda<m(x) \textup{ for every } x\in\Stsp,
\\[1mm]
\{z\in\Zsp\colon z>u\},
&
\textup{if } u<\infty \textup{ and }
\lambda>m(x) \textup{ for every } x\in\Stsp,
\\[1mm]
\{\infty\},
&
\textup{if } u=\infty,
\ \infty\in\Zsp,
\textup{ and }
\lambda>m(x) \textup{ for every } x\in\Stsp .
\end{cases}
\end{equation}
When \(u=\infty\) and \(\infty\notin\Zsp\), the final case cannot occur for
finite \(\lambda\), by \textup{(PCLI2)}.

\begin{definition}[Generalized inverse of the MP index]
\label{def:gim}
A threshold map \(\zeta\colon\mathbb R\to\Zsp\) is a \emph{generalized
inverse} of \(m\) if \(\zeta(\lambda)\in\Zsp_\lambda\) for every
\(\lambda\in\mathbb R\).
\end{definition}

\begin{lemma}[Generalized inverse order properties]
\label{lma:geninvprop}
Assume \textup{(PCLI2)}, and let \(\zeta\) be a generalized inverse of
\(m\). Then, for every \(x\in\Stsp\) and \(\lambda\in\mathbb R\),
\[
x\le\zeta(\lambda)\Longrightarrow m(x)\le\lambda,
\qquad
x\ge\zeta(\lambda)\Longrightarrow m(x)\ge\lambda.
\]
Moreover, if \(\lambda_1,
\lambda_2\in m(\Stsp)\) and
\(\lambda_1<\lambda_2\), then
\(\zeta(\lambda_1)<\zeta(\lambda_2)\).
\end{lemma}

\begin{proof}
If \(m(\zeta(\lambda))=\lambda\), the implications follow from monotonicity
of \(m\). In the lower boundary case of \eqref{eq:tslambm},
\(\zeta(\lambda)<\ell\), so the first implication is vacuous and the second
follows from \(\lambda<m(y)\) for all \(y\in\Stsp\). The upper boundary
cases are analogous. For the final claim, if
\(\zeta(\lambda_1)\ge\zeta(\lambda_2)\), then monotonicity gives
\(\lambda_1=m(\zeta(\lambda_1))\ge
m(\zeta(\lambda_2))=\lambda_2\), a contradiction.
\end{proof}

\subsection{Verification and characterization theorem for threshold-indexability}
\label{s:pclbsic}

The discounted real-state theorem in \cite{nmmor20} was stated as a
sufficient condition. In the present formulation the structure is sharper:
once \textup{(PCLI1)} and the marginal integration-by-parts identity
\textup{(PCLI3b)} are available, the MP-index monotonicity/continuity
condition \textup{(PCLI2)} is necessary and sufficient for
threshold-indexability.

We first record two sign lemmas. Write
\(v_{\lambda,x}(z)=f_x(z)-\lambda g_x(z)\), with the limiting interpretation
at \(z=\infty\) when \(\infty\in\Zsp\).

\begin{lemma}[Marginal sign identity]
\label{lem:pcli-marginal-sign}
Assume \textup{(PCLI1)}, \textup{(PCLI2)}, and \textup{(PCLI3b)}. Then, for
every \(x\in\Stsp\) and \(z\in\mathbb R\),
\begin{equation}
\label{eq:pcli-marginal-sign}
\sgn\bigl(f_x(z)-m(z)g_x(z)\bigr)
=
\sgn\bigl(m(x)-m(z)\bigr).
\end{equation}
\end{lemma}

\begin{proof}
Let \(D_x(z)\triangleq f_x(z)-m(z)g_x(z)\). If \(z<x\), then
\textup{(PCLI3b)} gives
\(D_x(z)=\int_{(z,x)}g_x(y)\,\mathrm dm(y)\). Since \(y\neq x\) throughout
\((z,x)\), \textup{(PCLI1)} gives \(g_x(y)>0\); hence the integral is
positive exactly when \(\mu_m((z,x))=m(x)-m(z)>0\).

If \(x\le z\), then
\(D_x(z)=-\int_{[x,z]}g_x(y)\,\mathrm dm(y)\). The only possible zero of
\(g_x\) on \([x,z]\) is at \(y=x\) with \(x\in\Exsp\), and
\(\mathrm dm\) has no atom there because \(m\) is continuous. Hence the
integral without the minus sign is positive exactly when
\(\mu_m([x,z])=m(z)-m(x)>0\). This proves \eqref{eq:pcli-marginal-sign}.
\end{proof}

\begin{lemma}[Sign at a generalized inverse threshold]
\label{lma:sign_mxz_lambda}
Assume Assumption~\ref{ass:regfg}, \textup{(PCLI1)}, \textup{(PCLI2)}, and
\textup{(PCLI3b)}. If \(z\in\Zsp_\lambda\), then, for every
\(x\in\Stsp\),
\begin{equation}
\label{eq:sgnsmxzlmxl}
\sgn\bigl(v_{\lambda,x}(z)\bigr)
=
\sgn\bigl(m(x)-\lambda\bigr).
\end{equation}
When \(z=\infty\in\Zsp\), this identity is understood with
\(v_{\lambda,x}(\infty)=f_x(\infty)-\lambda g_x(\infty)\).
\end{lemma}

\begin{proof}
If \(z\) is finite and \(m(z)=\lambda\), the result is
Lemma~\ref{lem:pcli-marginal-sign}. In the lower boundary case of
\eqref{eq:tslambm}, \(z<\ell\), \(m(z)=m(\ell)\), and
\(\lambda<m(y)\) for all \(y\in\Stsp\). Lemma~\ref{lem:pcli-marginal-sign}
gives \(f_x(z)-m(z)g_x(z)\ge0\), while \(m(z)-\lambda>0\) and
\(g_x(z)>0\) by \textup{(PCLI1)}; hence
\[
 v_{\lambda,x}(z)
 =\bigl[f_x(z)-m(z)g_x(z)\bigr]
  +\bigl[m(z)-\lambda\bigr]g_x(z)>0,
\]
which agrees with \(m(x)-\lambda>0\).

In the finite upper boundary case, \(u<\infty\), \(z>u\),
\(m(z)=m(u)\), and \(\lambda>m(y)\) for all \(y\in\Stsp\). The same sign
identity gives \(f_x(z)-m(z)g_x(z)\le0\), and \(g_x(z)>0\), so
\[
 v_{\lambda,x}(z)
 =\bigl[f_x(z)-m(z)g_x(z)\bigr]
  -\bigl[\lambda-m(z)\bigr]g_x(z)<0,
\]
matching \(m(x)-\lambda<0\).

Finally let \(z=\infty\in\Zsp\). If \(m(\infty)=\lambda\), the limiting
form of \textup{(PCLI3b)} gives
\[
\sgn\bigl(f_x(\infty)-m(\infty)g_x(\infty)\bigr)
=
\sgn\bigl(m(x)-m(\infty)\bigr),
\]
because \(g_x>0\) \(\mu_m\)-a.e. on \([x,\infty)\). If
\(m(\infty)\neq\lambda\) and \(\infty\in\Zsp_\lambda\), then the infinite
upper boundary case in \eqref{eq:tslambm} applies; hence
\(\lambda>m(y)\) for all \(y\in\Stsp\) and necessarily
\(\lambda>m(\infty)\). Since \(g_x(\infty)>0\) by \textup{(PCLI1)},
\[
 v_{\lambda,x}(\infty)
 =\bigl[f_x(\infty)-m(\infty)g_x(\infty)\bigr]
  -\bigl[\lambda-m(\infty)\bigr]g_x(\infty)<0,
\]
which agrees with \(m(x)-\lambda<0\).
\end{proof}

\begin{theorem}[Verification and characterization theorem for threshold-indexability]
\label{the:pcliii}
Assume Assumptions~\ref{ass:excharcosdp} and~\ref{ass:regfg}, and assume
\textup{(PCLI1)}. Then the following hold.

\begin{enumerate}[label=\textup{(\alph*)},leftmargin=*]

\item
\textup{Verification.} If \textup{(PCLI2)} and \textup{(PCLI3b)} hold for a
candidate MP index \(m\), then the project is threshold-indexable. Its
Whittle index is \(m\), with the exceptional-state extension specified in
\textup{(PCLI2)}, and its optimal threshold maps are precisely the
generalized inverses of \(m\).

\item
\textup{Necessity of \textup{(PCLI2)}.} Conversely, if the project is
threshold-indexable with Whittle index \(\lambda^*\), then
\textup{(PCLI2)} holds with \(m=\lambda^*\). In particular,
\[
        \lambda^*(x)=f_x(x)/g_x(x),
        \qquad x\in\Rgsp,
\]
while at exceptional states the extension values are the corresponding values
of the continuous index \(\lambda^*\).

\item
\textup{Characterization under \textup{(PCLI3b)}.} Consequently, under the
maintained assumptions and in any model where \textup{(PCLI3b)} has been
verified for a candidate MP index \(m\), the project is threshold-indexable
with Whittle index \(m\) iff \textup{(PCLI2)} holds for \(m\). When these
equivalent conditions hold, the optimal threshold maps are precisely the
generalized inverses of \(m\).

\end{enumerate}
\end{theorem}

\begin{proof}
For part \textup{(a)}, let \(\zeta\) be a generalized inverse of \(m\), fix
\(\lambda\), and put \(z=\zeta(\lambda)\in\Zsp_\lambda\). By
Lemma~\ref{lma:sign_mxz_lambda},
\[
\sgn v_{\lambda,x}(z)=\sgn\bigl(m(x)-\lambda\bigr),
\qquad x\in\Stsp .
\]
By Lemma~\ref{lma:geninvprop}, \(x\le z\) implies \(m(x)\le\lambda\), and
\(x\ge z\) implies \(m(x)\ge\lambda\). Therefore
\[
        v_{\lambda,x}(z)\le0 \quad\textup{on }S^0(z),
        \qquad
        v_{\lambda,x}(z)\ge0 \quad\textup{on }S^1(z).
\]
Assumption~\ref{ass:excharcosdp}\textup{(iii)} gives that the \(z\)-policy
is \(P_\lambda\)-optimal. Choosing
\(v_\lambda^*(x)=v_{\lambda,x}(z)\), again by
Assumption~\ref{ass:excharcosdp}\textup{(iii)}, yields
\[
        \sgn v_\lambda^*(x)
        =
        \sgn\bigl(m(x)-\lambda\bigr),
        \qquad x\in\Stsp,
\]
so the project is indexable with Whittle index \(m\). The same signs and
Lemma~\ref{lma:geninvprop} give the thresholdability inequalities; hence
every generalized inverse is an optimal threshold map.

It remains to prove that there are no other optimal threshold maps. Let
\(z^*(\lambda)\in\Zsp\) be any optimal threshold value. Since the Whittle
index is \(m\), Definition~\ref{def:potp} gives
\[
x\le z^*(\lambda)\Longrightarrow m(x)\le\lambda,
\qquad
x\ge z^*(\lambda)\Longrightarrow m(x)\ge\lambda .
\]
If \(z^*(\lambda)\in\Stsp\), taking \(x=z^*(\lambda)\) gives
\(m(z^*(\lambda))=\lambda\). If \(z^*(\lambda)<\ell\), the second
implication gives \(m(x)\ge\lambda\) for all \(x\); hence either
\(\lambda=m(\ell)\), in which case constant continuation puts
\(z^*(\lambda)\) in the level set, or \(\lambda<m(x)\) for all \(x\), the
lower boundary case of \eqref{eq:tslambm}. The finite and infinite upper
boundary cases are analogous. Thus \(z^*(\lambda)\in\Zsp_\lambda\), proving
that the optimal threshold maps are precisely the generalized inverses of
\(m\).

For part \textup{(b)}, suppose the project is threshold-indexable with
Whittle index \(\lambda^*\). We first show that \(\lambda^*\) is
nondecreasing. If \(x_1<x_2\) but
\(\lambda^*(x_1)>\lambda^*(x_2)\), choose
\(\lambda\in(\lambda^*(x_2),\lambda^*(x_1))\). Then action \(1\) is
uniquely optimal at \(x_1\) and action \(0\) is uniquely optimal at
\(x_2\), contradicting thresholdability: an optimal threshold either places
\(x_1\) in the passive region or places \(x_2\) in the active region.
Therefore \(\lambda^*\) is nondecreasing.

The same argument rules out jumps. If, at an interior point \(x_0\),
\(\lambda^*(x_0-)<\lambda^*(x_0)\), choose
\(\lambda\) strictly between these values. Then activation is uniquely
optimal at \(x_0\), whereas passivity is uniquely optimal at all
\(y<x_0\) sufficiently close to \(x_0\), which is incompatible with any
threshold. The case \(\lambda^*(x_0)<\lambda^*(x_0+)\) is symmetric, and
endpoints use the same one-sided argument. Thus \(\lambda^*\) is continuous
on \(\Stsp\).

If \(u=\infty\) and \(\infty\notin\Zsp\), then
\(\lambda^*(x)\to\infty\) as \(x\to\infty\). Otherwise, for some finite
\(\lambda>\sup_x\lambda^*(x)\), passivity would be uniquely optimal in every
state; thresholdability would then require an all-passive threshold, which no
finite threshold represents on the half-line.

Finally fix \(x\in\Rgsp\) and put \(\lambda=\lambda^*(x)\). By
indexability, both actions are optimal at \(x\). Since \(\lambda^*\) is
nondecreasing, the \(x\)-policy selects an optimal action in every state:
passive for \(y\le x\) and active for \(y>x\). Hence the \(x\)-policy is
\(P_\lambda\)-optimal by Assumption~\ref{ass:excharcosdp}\textup{(ii)}.
Assumption~\ref{ass:excharcosdp}\textup{(iii)} then gives
\[
0=v_{\lambda,x}(x)=f_x(x)-\lambda g_x(x).
\]
Because \(x\in\Rgsp\), \textup{(PCLI1)} gives \(g_x(x)>0\), and hence
\(\lambda^*(x)=f_x(x)/g_x(x)\). At exceptional states the diagonal ratio is
undefined because \(f_x(x)=g_x(x)=0\); the required values are those of the
continuous index \(\lambda^*\). Thus \textup{(PCLI2)} holds with
\(m=\lambda^*\).

Part \textup{(c)} follows immediately from parts \textup{(a)} and
\textup{(b)}.
\end{proof}

\begin{remark}[Why \textup{(PCLI3a)} and \textup{(PCLI3b)} are separated]
\label{re:mainthm}
Theorem~\ref{the:pcliii} uses \textup{(PCLI3b)}, not \textup{(PCLI3a)},
because \textup{(PCLI3b)} is the statewise sign identity needed to verify
optimal threshold policies. In the discounted real-state theory of
\cite{nmmor20}, the condition called \textup{(PCLI3)} was the
Radon--Nikodym identity \(\mathrm dF_x=m\,\mathrm dG_x\), and the marginal
identity could be derived using bounded variation of discounted marginal
metrics. That derivation need not survive the average limit, where
\(z\mapsto g_x(z)\) may fail to generate a signed LS measure. Therefore the
criterion-agnostic framework keeps \textup{(PCLI3a)} as the PCL and
shadow-price identity for threshold performance metrics, while imposing
\textup{(PCLI3b)} directly as the identity used in the verification proof.
Proposition~\ref{pro:intgradz} below shows that, under LS regularity,
\textup{(PCLI3a)} is nevertheless necessary once threshold-indexability is
known.
\end{remark}

\subsection{Geometric and economic interpretation}
\label{s:geom_interp}

The full PCL conditions have the usual resource--reward frontier
interpretation. This interpretation is not used in Theorem~\ref{the:pcliii}
and requires the standard randomized-policy convexification and compactness
assumptions, as in \cite[Online Companion, App.~B]{nmmor20}. Fix an initial
distribution \(\nu\) and write
\(F(\nu,
\pi)=\int_\Stsp F(x,
\pi)\,\nu(dx)\) and
\(G(\nu,
\pi)=\int_\Stsp G(x,
\pi)\,\nu(dx)\). The achievable
resource--reward region is
\(\mathcal R_{GF}(\nu)=\{(G(\nu,
\pi),F(\nu,
\pi)):\pi\in\Pi\}\), and its
efficient frontier is the upper reward boundary at each resource level.

Under threshold-indexability, randomized threshold policies trace this
frontier. If \(z\) and \(z^{\scriptscriptstyle -}\) are the passive- and
active-at-threshold conventions, mixtures between their performance points
generate the corresponding frontier segment. The Radon--Nikodym identity
\textup{(PCLI3a)} gives
\(\mathrm dF_\nu(z)=m(z)\,\mathrm dG_\nu(z)\), whenever the LS measures are
well defined. At a threshold atom,
\[
        m(z)=
        \frac{F_\nu(z^{\scriptscriptstyle -})-F_\nu(z)}
             {G_\nu(z^{\scriptscriptstyle -})-G_\nu(z)},
\]
provided the denominator is positive. Thus \(m(z)\) is the supporting slope
of the efficient frontier and the resource shadow price at which the local
threshold exchange is neutral.

\subsection{Why the conditions are natural}
\label{s:potip}

The theorem above gives the formal characterization. We briefly record why
its conditions are natural and how they relate to the discounted real-state
PCL framework of \cite{nmmor20}.

\begin{lemma}[Whittle index and thresholdability; cf.\ Lemma~2 in \cite{nmmor20}]
\label{lma:iwrttpndlc}
Suppose the project is indexable with Whittle index \(\lambda^*\). Extend
\(\lambda^*\) to finite thresholds by constant continuation outside the state
space:
\[
        \lambda^*(z)=\lambda^*(\ell), \quad z<\ell,
        \qquad
        \lambda^*(z)=\lambda^*(u), \quad z\ge u,
\ u<\infty .
\]
If \(u=\infty\) and \(\infty\in\Zsp\), set
\(\lambda^*(\infty)\triangleq\lim_{x\to\infty}\lambda^*(x)\), when the
limit exists. Then the project is thresholdable iff \(\lambda^*\) is
nondecreasing and continuous on \(\Stsp\), and, when
\(u=\infty\) and \(\infty\notin\Zsp\),
\(\lambda^*(x)\to\infty\) as \(x\to\infty\). In that case, the optimal
threshold maps are the generalized inverses of \(\lambda^*\).
\end{lemma}

\begin{proof}
This is the criterion-agnostic form of \cite[Lemma~2]{nmmor20}. The proof
uses only the sign characterization of indexability, the thresholdability
inequalities, and the finite-threshold continuation convention; the same
argument is given in the necessity part of Theorem~\ref{the:pcliii}.
\end{proof}

Thus \textup{(PCLI2)} is not an artifact of the PCL method: it is exactly the
state-monotonicity structure imposed on any Whittle index by
thresholdability.

\begin{proposition}[MP ratio as candidate index; cf.\ Proposition~1 in \cite{nmmor20}]
\label{pro:phiwphimpi}
Suppose Assumption~\ref{ass:excharcosdp} holds. If the project is
threshold-indexable with Whittle index \(\lambda^*\), then, for every
\(x\in\Stsp\),
\[
        f_x(x)-\lambda^*(x)g_x(x)=0,
\qquad
        f_x(x^{\scriptscriptstyle -})
        -
        \lambda^*(x)g_x(x^{\scriptscriptstyle -})=0 .
\]
Consequently, whenever \(g_x(x)\neq0\),
\(\lambda^*(x)=f_x(x)/g_x(x)=m_x(x)\). Under the everywhere-positive version
of \textup{(PCLI1)}, the Whittle index must therefore coincide with the MP
index.
\end{proposition}

\begin{proof}
Put \(\lambda=\lambda^*(x)\). By Lemma~\ref{lma:iwrttpndlc},
\(\lambda^*\) is nondecreasing, so the \(x\)-policy and the
\(x^{\scriptscriptstyle -}\)-policy each select an optimal action in every
state at price \(\lambda\). Assumption~\ref{ass:excharcosdp} then allows the
optimal Q-value difference at state \(x\) to be represented by either
threshold continuation. Since \(\lambda=\lambda^*(x)\), indexability gives
\(v_\lambda^*(x)=0\), yielding both displayed identities. Division by
\(g_x(x)\) gives the MP-ratio formula when \(g_x(x)\neq0\).
\end{proof}

If \(g_x(x)=0\), the same identity gives \(f_x(x)=0\), creating a removable
\(0/0\) diagonal singularity. This is the reason for the exceptional-set
formulation in \textup{(PCLI1)}--\textup{(PCLI2)}.

\begin{proposition}[Necessity of the Radon--Nikodym identity; cf.\ Proposition~2 in \cite{nmmor20}]
\label{pro:intgradz}
Suppose the project is threshold-indexable with Whittle index \(\lambda^*\),
extended to finite thresholds by constant continuation outside \(\Stsp\).
Assume that \(F_x\) and \(G_x\) satisfy the LS regularity in
Assumption~\ref{ass:regFG}. Then, for every \(x\in\Stsp\) and finite
\(z_1<z_2\),
\[
        F_x(z_2)-F_x(z_1)
        =
        \int_{(z_1,z_2]}\lambda^*(z)\,\mathrm dG_x(z).
\]
Equivalently,
\(\mathrm dF_x(z)=\lambda^*(z)\,\mathrm dG_x(z)\) as a signed LS measure
identity.
\end{proposition}

\begin{proof}
This is the criterion-agnostic form of \cite[Proposition~2]{nmmor20}. The
proof uses convexity of \(\lambda\mapsto V_\lambda^*(x)\), the subgradient
identity that identifies \(-G_x(z^*(\lambda))\) along optimal threshold maps,
and the Stieltjes change of variables from price \(\lambda\) to threshold
\(z\) through \(\lambda=\lambda^*(z)\). These steps use
threshold-indexability and the LS regularity of \(F_x\) and \(G_x\), not the
discounted nature of the performance criterion.
\end{proof}

Once \(\lambda^*\) is identified with the MP index \(m\),
Proposition~\ref{pro:intgradz} becomes \textup{(PCLI3a)}. The marginal
identity \textup{(PCLI3b)} has a different role: it is not a convexity
consequence and is not an LS identity with integrator \(\mathrm dg_x\). It is
the stable sign mechanism
\[
\sgn\bigl(f_x(z)-m(z)g_x(z)\bigr)
=
\sgn\bigl(m(x)-m(z)\bigr),
\]
which is robust in long-run average scalar-state models where
\(z\mapsto g_x(z)\) may fail to have locally bounded variation.

\section{Vanishing-discount transfer to the discrete-time long-run average criterion}
\label{s:dtla}

This section transfers discounted PCL-indexability to the discrete-time
long-run average criterion. We consider a scalar-state project for which the
\(\beta\)-discounted PCL-indexability conditions of \cite{nmmor20} hold for
all \(\beta\in(0,1)\). The aim is to identify assumptions under which the
vanishing-discount limits of the discounted threshold metrics satisfy the
average-criterion PCL conditions in Definition~\ref{def:pcli}.

The main issue is the limiting behavior of marginal metrics as functions of the
threshold. In the discounted theory, \(z\mapsto g_{\beta,x}(z)\) has bounded
variation, and signed LS integrals involving \(\mathrm dg_{\beta,x}\) are
available. In the long-run average limit, \(z\mapsto g_x(z)\) may fail to have
locally bounded variation, for example in scalar models with threshold-orbit
accumulation. We therefore transfer the integration-by-parts form
\textup{(PCLI3b)}, where \(g_x\) appears only as an integrand with respect to
\(\mathrm dm\), the nonnegative LS measure generated by the limiting MP index.

The dynamic-programming component is standard. It parallels the classical
vanishing-discount approach to average-cost MDPs; see, for example,
\cite{schal93} and \cite[Ch.~5]{herlerLass96}. Weakly continuous Borel-state
MDP theory gives common sufficient conditions for the average optimality
relations used below: \cite{feinbergetal12} treats average-cost optimality
inequalities and stationary average-optimal policies, while
\cite{feinbergLiang22} gives conditions leading to the average-cost optimality
equation. We state these optimality prerequisites separately from the
metric-limit hypotheses needed for PCL-indexability.

\subsection{Average \texorpdfstring{\(\lambda\)}{lambda}-price problems and Q-values}
\label{ss:av-lambda}

Consider a discrete-time controlled Markov project on \(\Stsp\), with action
set \(\{0,1\}\), Borel transition kernels \(\kappa^a(x,\cdot)\), one-step
reward \(r(x,a)\), and one-step resource consumption \(c(x,a)\). Write
\(\kappa_x^a(\cdot)\triangleq \kappa^a(x,\cdot)\). Thus \(\kappa_x^a\) is a
probability measure on \(\Stsp\), and transition expectations are denoted by
\(\int_\Stsp h(y)\,\kappa_x^a(dy)\). This notation is reserved for transition
probabilities, whereas \(\int\varphi(z)\,\mathrm dH(z)\) denotes an LS integral
generated by a threshold-performance function. All displayed integrals are
assumed well defined.

For any measurable \(h\colon\Stsp\to\mathbb R\) for which the integral is well
defined, set
\[
        ({\mathcal K}^a h)(x)
        \triangleq
        \int_\Stsp h(y)\,\kappa_x^a(dy),
        \qquad x\in\Stsp,\quad a\in\{0,1\}.
\]
For price \(\lambda\in\mathbb R\), let
\(r_\lambda(x,a)\triangleq r(x,a)-\lambda c(x,a)\). Let \(\Pi\) be the class of
admissible policies, and let \(\mathbb P_x^\pi\), \(\mathbb E_x^\pi\) denote
probability and expectation for the controlled chain started from \(X_0=x\).
For policies with well-defined finite-horizon expectations and no
\(\infty-\infty\) ambiguity, define the average net value
\begin{equation}
\label{eq:avgJlambda}
V_\lambda(x,\pi)
\triangleq
\liminf_{T\to\infty}
\frac{1}{T}\,
\mathbb E_x^\pi
\left[
\sum_{t=0}^{T-1} r_\lambda(X_t,A_t)
\right],
\qquad x\in\Stsp .
\end{equation}
In general this value may depend on \(x\). In the transfer argument, threshold
policies are singled out: Assumption~\ref{ass:avg-limits} requires their
separate long-run average reward and resource gains to exist and be
state-independent. These gains are denoted by \(F(z)\) and \(G(z)\), so for the
\(z\)-threshold policy \(\pi^z\),
\[
        V_\lambda(x,\pi^z)=F(z)-\lambda G(z),
        \qquad x\in\Stsp,
\]
whenever the separate averages exist. No such state-independence is imposed on
arbitrary policies.

The optimal average net value is
\(V_\lambda^*(x)\triangleq\sup_{\pi\in\Pi}V_\lambda(x,\pi)\), where the supremum
is taken over policies for which \eqref{eq:avgJlambda} is well defined. The
average \(\lambda\)-price problem \(P_\lambda^{\rm av}\) is to find
\(\pi_\lambda^*\in\Pi\) with
\(V_\lambda(x,\pi_\lambda^*)=V_\lambda^*(x)\) for all \(x\in\Stsp\).

For a measurable function \(h\colon\Stsp\to\mathbb R\), define the one-step net
Q-value
\[
Q_\lambda^h(x,a)
\triangleq
r_\lambda(x,a)+({\mathcal K}^a h)(x),
\qquad a\in\{0,1\}.
\]
The next assumption supplies the average optimality structure needed to recover
Assumption~\ref{ass:excharcosdp}.

\begin{assumption}[Average optimality and optimal-action characterization]
\label{ass:avg-acoe}
For each \(\lambda\in\mathbb R\), the following hold.
\begin{enumerate}[label=\textup{(\roman*)},leftmargin=*]

\item
There exist a state-independent gain \(\rho_\lambda\in\mathbb R\) and a
measurable bias \(h_\lambda\colon\Stsp\to\mathbb R\) satisfying
\begin{equation}
\label{eq:avg-acoe-main}
\rho_\lambda+h_\lambda(x)
=
\max_{a\in\{0,1\}}Q_\lambda^{h_\lambda}(x,a),
\qquad x\in\Stsp .
\end{equation}
Every SD policy selecting maximizers in \eqref{eq:avg-acoe-main} is
\(P_\lambda^{\rm av}\)-optimal, and \(V_\lambda^*(x)=\rho_\lambda\) for every
\(x\in\Stsp\).

\item
If an SD policy \(\pi\) has an average net gain \(\rho_\lambda^\pi\) and a
measurable bias \(h_\lambda^\pi\) satisfying
\[
\rho_\lambda^\pi+h_\lambda^\pi(x)
=
Q_\lambda^{h_\lambda^\pi}(x,\pi(x)),
\qquad x\in\Stsp,
\]
then \(\pi\) is \(P_\lambda^{\rm av}\)-optimal iff
\[
\pi(x)\in\operatorname*{argmax}_{a\in\{0,1\}}
Q_\lambda^{h_\lambda^\pi}(x,a),
\qquad x\in\Stsp .
\]
\end{enumerate}
\end{assumption}

Choose a bias \(h_\lambda\) satisfying \eqref{eq:avg-acoe-main} and define
\begin{equation}
\label{eq:Vlambastarav}
Q_\lambda^*(x,a)
\triangleq
Q_\lambda^{h_\lambda}(x,a),
\qquad a\in\{0,1\}.
\end{equation}
Set \(v_\lambda^*(x)\triangleq Q_\lambda^*(x,1)-Q_\lambda^*(x,0)\). The
additive normalization of \(h_\lambda\) cancels in this difference.
Assumption~\ref{ass:avg-acoe}\textup{(i)} then gives the average-criterion
counterpart of Assumption~\ref{ass:excharcosdp}: action \(1\) is optimal iff
\(v_\lambda^*(x)\ge0\), and action \(0\) is optimal iff
\(v_\lambda^*(x)\le0\).

\begin{remark}[Verifying Assumption~\ref{ass:avg-acoe}]
\label{re:avg-acoe}
Average-cost MDP results can be used by rewriting reward maximization as cost
minimization with one-period cost
\[
        \tilde c_\lambda(x,a)
        \triangleq
        \lambda c(x,a)-r(x,a)
        =-r_\lambda(x,a).
\]
The resulting average-cost optimality equation is then rewritten in the reward
form \eqref{eq:avg-acoe-main}. Part \textup{(ii)} is the corresponding
policy-specific optimal-action characterization: when a policy has its own
gain--bias pair, optimality is equivalent to selecting a maximizing one-step
Q-action in every state.
\end{remark}

We shall use the following compactness criterion in the web-crawling and
noisy-channel examples. Vanishing-discount proofs of average optimality
equations commonly require relative compactness of normalized discounted value
functions; see \cite{schal93,feinbergetal12,feinbergLiang22} and
\cite[Ch.~5]{herlerLass96}. Invariant Lipschitz classes provide a convenient
route to that compactness; see, for example, \cite{hinderer05}.

For \(0\le L<\infty\), consider the class of \emph{\(L\)-Lipschitz functions}, 
\[
\mathrm{Lip}_{L}(\Stsp)
\triangleq
\{v\colon\Stsp\to\mathbb R:
        |v(x)-v(y)|\le L|x-y|,
        \ x,y\in\Stsp\}.
\]
For \(\beta\in(0,1)\), define the discounted net-reward Bellman operator
\[
(\mathcal T_{\beta,\lambda}v)(x)
\triangleq
\max_{a\in\{0,1\}}
\left\{
        r_\lambda(x,a)+\beta({\mathcal K}^a v)(x)
\right\}.
\]
For a function class \(\mathcal C\), the notation
\(\mathcal T_{\beta,\lambda}(\mathcal C)\subseteq\mathcal C\) means that
\(\mathcal T_{\beta,\lambda}v\in\mathcal C\) for every \(v\in\mathcal C\). Let
\(V_{\beta,\lambda}^*\) be the optimal value function of the discounted
\(\lambda\)-price problem.

\begin{lemma}[Invariant Lipschitz class and relative discounted rewards]
\label{lem:avg-lip-relative}
Fix \(\lambda\in\mathbb R\). Suppose \(\Stsp=[\ell,u]\), with
\(\ell,u\in\mathbb R\), that \(r_\lambda\) is bounded, and that for some
\(L_\lambda<\infty\), independent of \(\beta\),
\[
        \mathcal T_{\beta,\lambda}
        \bigl(\mathrm{Lip}_{L_\lambda}(\Stsp)\bigr)
        \subseteq
        \mathrm{Lip}_{L_\lambda}(\Stsp),
        \qquad 0<\beta<1 .
\]
Then \(V_{\beta,\lambda}^*\) is \(L_\lambda\)-Lipschitz on \(\Stsp\). Hence
\[
h_{\beta,\lambda}(x)
\triangleq
V_{\beta,\lambda}^*(x)-\inf_{y\in\Stsp}V_{\beta,\lambda}^*(y)
\]
is \(L_\lambda\)-Lipschitz and satisfies
\[
        0\le h_{\beta,\lambda}(x)
        \le L_\lambda(u-\ell),
        \qquad x\in\Stsp,
        \quad 0<\beta<1 .
\]
Thus \(\{h_{\beta,\lambda}:\beta\in(0,1)\}\) is uniformly bounded and
equicontinuous on \(\Stsp\).
\end{lemma}

\begin{proof}
The zero function belongs to \(\mathrm{Lip}_{L_\lambda}(\Stsp)\), so all
value-iteration iterates \(\mathcal T_{\beta,\lambda}^{\,n}0\) are
\(L_\lambda\)-Lipschitz. Since \(r_\lambda\) is bounded and \(\beta<1\),
\(\mathcal T_{\beta,\lambda}\) is a sup-norm contraction on bounded functions,
and the iterates converge uniformly to \(V_{\beta,\lambda}^*\). The Lipschitz
class is closed under uniform limits. Subtracting a constant preserves the
Lipschitz constant, and boundedness of \(\Stsp\) gives the displayed bound; the
common Lipschitz constant gives equicontinuity.
\end{proof}

\subsection{Vanishing-discount limits of project metrics}
\label{ss:av-limits}

For \(\beta\in(0,1)\), let
\(F_{\beta,x}(z),G_{\beta,x}(z),f_{\beta,x}(z),g_{\beta,x}(z)\) be the
discounted threshold performance and marginal metrics. Define the discounted MP
metric and index by
\[
        m_{\beta,x}(z)
        \triangleq
        \frac{f_{\beta,x}(z)}{g_{\beta,x}(z)},
        \qquad
        m_\beta(x)\triangleq m_{\beta,x}(x),
\]
where the ratio is used on its positive-resource domain.

\begin{assumption}[Discounted PCL-indexability]
\label{ass:disc-pcli}
For every \(\beta\in(0,1)\), the \(\beta\)-discounted project is
PCL-indexable with respect to threshold policies in the sense of
\cite[Definition~7]{nmmor20}, with discounted threshold metrics
\(F_{\beta,x},G_{\beta,x}\), discounted marginal metrics
\(f_{\beta,x},g_{\beta,x}\), and discounted MP index \(m_\beta\).
\end{assumption}

Under Assumption~\ref{ass:disc-pcli}, the fixed-discount results of
\cite{nmmor20} apply. In particular, the threshold and marginal metrics are
bounded c\`adl\`ag functions of the threshold, active-at-threshold values are
left limits, endpoint values are one-sided limits, \(g_{\beta,x}(z)>0\), the
index \(m_\beta\) is nondecreasing and continuous, and the discounted
Radon--Nikodym PCL identity holds:
\begin{equation}
\label{eq:disc-pcl-FG}
F_{\beta,x}(z_2)-F_{\beta,x}(z_1)
=
\int_{(z_1,z_2]}m_\beta(y)\,\mathrm dG_{\beta,x}(y),
\qquad z_1<z_2 .
\end{equation}
We also use the integration-by-parts form of the discounted marginal identity
from \cite[Lemma~22]{nmmor20}. With
\(D_{\beta,x}(z)\triangleq f_{\beta,x}(z)-m_\beta(z)g_{\beta,x}(z)\), for finite
\(z\),
\begin{equation}
\label{eq:disc-marginal-ibp}
D_{\beta,x}(z)
=
\begin{cases}
\displaystyle
\int_{(z,x)}g_{\beta,x}(y)\,\mathrm dm_\beta(y), & z<x,\\[2.5ex]
\displaystyle
-\int_{[x,z]}g_{\beta,x}(y)\,\mathrm dm_\beta(y), & x\le z,
\end{cases}
\end{equation}
where \(\mathrm dm_\beta\) is the nonnegative LS measure generated by
\(m_\beta\). This is the form that survives the average limit because the
marginal resource appears as an integrand, not as an LS integrator.

Since the average reward and resource gains of threshold policies are
state-independent under the hypotheses below, the average threshold metrics in
the notation of \S~\ref{s:pcli} are
\[
        F_x(z)\equiv F(z),
        \qquad
        G_x(z)\equiv G(z).
\]
Let \(h_F\) and \(h_G\) denote the associated bias-type threshold metrics, and
write \(a_z(x)\triangleq\1_{\{x>z\}}\). Define
\[
        \Delta r(x)\triangleq r(x,1)-r(x,0),
        \qquad
        \Delta c(x)\triangleq c(x,1)-c(x,0),
\]
and, for a measurable \(h\colon\Stsp\to\mathbb R\),
\[
        \Delta_\kappa h(x)
        \triangleq
        ({\mathcal K}^1h)(x)-({\mathcal K}^0h)(x).
\]
For threshold-dependent functions,
\(\Delta_\kappa h(\cdot,z)(x)\triangleq
({\mathcal K}^1 h(\cdot,z))(x)-({\mathcal K}^0 h(\cdot,z))(x)\). The limiting
average marginal metrics are
\begin{align}
\label{eq:avg-f-from-hF}
f_x(z)
&\triangleq
\Delta r(x)+\Delta_\kappa h_F(\cdot,z)(x),\\
\label{eq:avg-g-from-hG}
g_x(z)
&\triangleq
\Delta c(x)+\Delta_\kappa h_G(\cdot,z)(x).
\end{align}

\begin{assumption}[Vanishing-discount transfer hypotheses]
\label{ass:avg-limits}
There exist gain metrics \(F,G\colon\Zsp\to\mathbb R\), bias metrics
\(h_F,h_G\colon\Stsp\times\Zsp\to\mathbb R\), a possibly empty Borel set
\(\Exsp\subseteq\Stsp\), and a finite-valued function
\(m\colon\Stsp\to\mathbb R\) satisfying the conditions below. The functions
\(m_\beta\) and \(m\) are extended to finite thresholds outside \(\Stsp\) by
endpoint constants.

\begin{enumerate}[label=\textup{(\roman*)},leftmargin=*]

\item
For every \(x\in\Stsp\) and \(z\in\Zsp\),
\[
\bigl((1-\beta)F_{\beta,x}(z),(1-\beta)G_{\beta,x}(z)\bigr)
\to
(F(z),G(z)),
\qquad \beta\uparrow1,
\]
and
\[
\left(
F_{\beta,x}(z)-\frac{F(z)}{1-\beta},
G_{\beta,x}(z)-\frac{G(z)}{1-\beta}
\right)
\to
(h_F(x,z),h_G(x,z)),
\qquad \beta\uparrow1.
\]
The gains \(F(z)\) and \(G(z)\) are independent of \(x\). With
\[
        H^F_{\beta,z}(x)
        \triangleq
        F_{\beta,x}(z)-\frac{F(z)}{1-\beta},
        \qquad
        H^G_{\beta,z}(x)
        \triangleq
        G_{\beta,x}(z)-\frac{G(z)}{1-\beta},
\]
the bias convergence is compatible with one-step transition expectations:
for every \(a\in\{0,1\}\), \(x\in\Stsp\), and \(z\in\Zsp\),
\[
        ({\mathcal K}^a H^F_{\beta,z})(x)
        \to
        ({\mathcal K}^a h_F(\cdot,z))(x),
        \qquad
        ({\mathcal K}^a H^G_{\beta,z})(x)
        \to
        ({\mathcal K}^a h_G(\cdot,z))(x).
\]

\item
With \(f_x\) and \(g_x\) defined by
\eqref{eq:avg-f-from-hF}--\eqref{eq:avg-g-from-hG},
\[
        g_x(z)>0
\]
whenever either \(z\neq x\) or
\(x\in\Rgsp\triangleq\Stsp\setminus\Exsp\), while
\[
        g_x(x)=f_x(x)=0,
        \qquad x\in\Exsp .
\]

\item
The indices \(m_\beta\) converge locally uniformly to \(m\) on compact
finite-threshold intervals. The limit \(m\) is nondecreasing and continuous on
\(\Stsp\). On \(\Rgsp\),
\[
        m(x)=\frac{f_x(x)}{g_x(x)},
\]
whereas on \(\Exsp\) it is the specified finite-valued, nondecreasing,
continuous extension. If \(u=\infty\) and \(\infty\notin\Zsp\), then
\(m(x)\to\infty\) as \(x\to\infty\).

\item
With constant continuation outside the effective state range, the limiting
resource gain \(G\colon\mathbb R\to\mathbb R\) is finite-valued, c\`adl\`ag,
nonincreasing, and not identically constant. Its active-at-threshold values are
left limits, and, if \(\infty\in\Zsp\), \(G(\infty)\) is the corresponding
finite one-sided limit. If \(\infty\in\Zsp\), \(F(\infty)\) is also finite and
is the corresponding one-sided limit.

For each \(x\in\Stsp\), the maps \(z\mapsto f_x(z)\) and \(z\mapsto g_x(z)\)
are finite-valued, Borel measurable, and c\`adl\`ag, with
active-at-threshold values given by left limits and finite limits at
\(\infty\) when \(\infty\in\Zsp\). No locally bounded variation assumption is
imposed on these marginal maps.

\item
For every finite \(z_1<z_2\) and every \(x\in\Stsp\),
\begin{equation}
\label{eq:vd-gain}
\lim_{\beta\uparrow1}
(1-\beta)\int_{(z_1,z_2]}m_\beta(y)\,\mathrm dG_{\beta,x}(y)
=
\int_{(z_1,z_2]}m(y)\,\mathrm dG(y).
\end{equation}
For every \(x\in\Stsp\) and finite \(z\),
\begin{align}
\label{eq:vd-marginal-pos}
\lim_{\beta\uparrow1}
\int_{(z,x)}g_{\beta,x}(y)\,\mathrm dm_\beta(y)
&=
\int_{(z,x)}g_x(y)\,\mathrm dm(y),
&& z<x,\\
\label{eq:vd-marginal-neg}
\lim_{\beta\uparrow1}
\int_{[x,z]}g_{\beta,x}(y)\,\mathrm dm_\beta(y)
&=
\int_{[x,z]}g_x(y)\,\mathrm dm(y),
&& x\le z.
\end{align}
Here \(\mathrm dm_\beta\) and \(\mathrm dm\) are the nonnegative LS measures
generated by \(m_\beta\) and \(m\).

\end{enumerate}
\end{assumption}

The fixed-threshold Poisson equations follow from
Assumption~\ref{ass:avg-limits}\textup{(i)}. The discounted equations are
\[
F_{\beta,x}(z)
=
r(x,a_z(x))+
\beta({\mathcal K}^{a_z(x)}F_{\beta,\cdot}(z))(x),
\]
and similarly for \(G_{\beta,x}(z)\). Since \(F(z)\) and \(G(z)\) are constant
in the next-state variable, these equations become
\[
        F(z)+H^F_{\beta,z}(x)
        =
        r(x,a_z(x))+
        \beta({\mathcal K}^{a_z(x)}H^F_{\beta,z})(x),
\]
and
\[
        G(z)+H^G_{\beta,z}(x)
        =
        c(x,a_z(x))+
        \beta({\mathcal K}^{a_z(x)}H^G_{\beta,z})(x).
\]
Letting \(\beta\uparrow1\) gives
\begin{align}
\label{eq:avg-threshold-poisson-F}
F(z)+h_F(x,z)
&=
r(x,a_z(x))+({\mathcal K}^{a_z(x)}h_F(\cdot,z))(x),\\
\label{eq:avg-threshold-poisson-G}
G(z)+h_G(x,z)
&=
c(x,a_z(x))+({\mathcal K}^{a_z(x)}h_G(\cdot,z))(x).
\end{align}

The limiting marginal metrics also follow from item \textup{(i)}. The
discounted marginal metrics satisfy
\[
        f_{\beta,x}(z)=
        \Delta r(x)+\beta\Delta_\kappa F_{\beta,\cdot}(z)(x),
        \qquad
        g_{\beta,x}(z)=
        \Delta c(x)+\beta\Delta_\kappa G_{\beta,\cdot}(z)(x).
\]
The constant gain terms cancel in the action difference, so the centered bias
limits yield
\[
        (f_{\beta,x}(z),g_{\beta,x}(z))
        \to
        (f_x(z),g_x(z)),
        \qquad \beta\uparrow1,
\]
with \(f_x,g_x\) as in
\eqref{eq:avg-f-from-hF}--\eqref{eq:avg-g-from-hG}.

\begin{remark}
\label{re:avg-limits}
Assumption~\ref{ass:avg-limits} concerns only the vanishing-discount passage;
the fixed-discount PCL facts are inherited from \cite{nmmor20}. The functions
\(F\) and \(G\) are the long-run average reward and resource rates of threshold
policies, i.e., the leading \((1-\beta)^{-1}\) terms of the discounted
threshold metrics. The bias metrics \(h_F,h_G\) determine the average one-step
Q-metrics for threshold continuations, and additive normalizations cancel in
the marginal differences \eqref{eq:avg-f-from-hF}--\eqref{eq:avg-g-from-hG}.
The exceptional set allows \(g_x(x)\) to vanish at selected diagonal states; at
such states the MP ratio is replaced by the specified continuous extension,
often obtained as an Abelian vanishing-discount limit.

In applications, the essential checks are that the average limits have the
regularity required by the abstract framework and that the LS limit passages
\eqref{eq:vd-gain}--\eqref{eq:vd-marginal-neg} are valid. Local bounded
variation of \(f_x\) and \(g_x\) is not required, because the marginal identity
uses \(g_x\) as an integrand with respect to \(\mathrm dm\), not as an
integrator.
\end{remark}

\begin{remark}[Checking the LS limit passages]
\label{re:vd-criterion}
The gain passage \eqref{eq:vd-gain} can be verified on compact threshold
intervals by convergence of the signed LS measures generated by
\((1-\beta)G_{\beta,x}\), including endpoint atoms, together with uniform local
total-variation control and locally uniform convergence \(m_\beta\to m\). The
marginal passages \eqref{eq:vd-marginal-pos}--\eqref{eq:vd-marginal-neg} are
different: the measures are generated by \(m_\beta\), not by
\(g_{\beta,x}\). Since \(m_\beta\to m\) locally uniformly and the
\(m_\beta\)'s are nondecreasing and continuous, \(\mathrm dm_\beta\) converges
weakly on compact intervals to \(\mathrm dm\). In concrete models this is often
combined with convergence of the integrands, dominated convergence when common
densities are available, or atomic/itinerary calculations when threshold
metrics are step functions.
\end{remark}

\subsection{Transfer theorem}
\label{ss:avg-transfer}

We now collect the consequences of the discounted PCL hypothesis and the
vanishing-discount limits: the limiting metrics satisfy the abstract PCL
conditions, and the average \(\lambda\)-price problems have the optimal-action
structure needed for Theorem~\ref{the:pcliii}.

\begin{theorem}[Transfer of discounted PCL-indexability to the average criterion]
\label{the:avg-pcli-transfer}
Assume Assumptions~\ref{ass:avg-acoe}, \ref{ass:disc-pcli}, and
\ref{ass:avg-limits}. Then the average threshold metrics
\(F_x(z)\equiv F(z)\) and \(G_x(z)\equiv G(z)\), together with the marginal
metrics \(f_x(z)\) and \(g_x(z)\), satisfy
Assumptions~\ref{ass:regFG} and~\ref{ass:regfg} and the PCL-indexability
conditions of Definition~\ref{def:pcli}, with MP index \(m\). Moreover, the
average \(\lambda\)-price problems satisfy the optimal-action characterization
required in Theorem~\ref{the:pcliii}. Consequently, the average-criterion
project is threshold-indexable. Its Whittle index is \(m\), with the
exceptional-state extension specified in
Assumption~\ref{ass:avg-limits}\textup{(iii)}, and its optimal threshold maps
are precisely the generalized inverses of \(m\).
\end{theorem}

\begin{proof}
Assumption~\ref{ass:avg-acoe}\textup{(i)} gives the optimal average Q-value
characterization. For a fixed threshold \(z\), set
\[
        h_{\lambda,z}(x)
        \triangleq
        h_F(x,z)-\lambda h_G(x,z).
\]
Combining \eqref{eq:avg-threshold-poisson-F} and
\eqref{eq:avg-threshold-poisson-G} yields
\[
F(z)-\lambda G(z)+h_{\lambda,z}(x)
=
r_\lambda(x,a_z(x))+
({\mathcal K}^{a_z(x)}h_{\lambda,z})(x).
\]
The corresponding one-step action difference is
\[
\begin{split}
&\left[r_\lambda(x,1)+({\mathcal K}^1h_{\lambda,z})(x)\right]
-
\left[r_\lambda(x,0)+({\mathcal K}^0h_{\lambda,z})(x)\right] \\
&\qquad = f_x(z)-\lambda g_x(z).
\end{split}
\]
Thus \(f_x\) and \(g_x\) are the average marginal Q-metrics for threshold
continuations. Assumption~\ref{ass:avg-acoe}\textup{(ii)} gives the
policy-specific optimal-action characterization for such continuations, which
is the structure used in the verification part of Theorem~\ref{the:pcliii}.

Assumption~\ref{ass:avg-limits}\textup{(ii)} gives \textup{(PCLI1)}, and item
\textup{(iii)} gives \textup{(PCLI2)}. Item \textup{(iv)} gives the
c\`adl\`ag regularity of \(G\) and of the marginal maps; since \(G\) is
finite-valued and nonincreasing, it has locally bounded variation on finite
threshold intervals. It remains to verify \textup{(PCLI3a)}, the corresponding
regularity of \(F\), and \textup{(PCLI3b)}.

Multiplying \eqref{eq:disc-pcl-FG} by \(1-\beta\) and letting
\(\beta\uparrow1\), using Assumption~\ref{ass:avg-limits}\textup{(i)} and
\eqref{eq:vd-gain}, gives, for finite \(z_1<z_2\),
\[
F(z_2)-F(z_1)
=
\int_{(z_1,z_2]}m(y)\,\mathrm dG(y).
\]
This is \textup{(PCLI3a)} for \(F_x(z)\equiv F(z)\) and
\(G_x(z)\equiv G(z)\). It also gives local bounded variation of \(F\): on a
compact finite interval \(I\),
\[
        \operatorname{TV}_I(F)
        \le
        \sup_{y\in I}|m(y)|\,\operatorname{TV}_I(G)<\infty .
\]
Right-continuity and left limits of \(F\) follow from the LS identity and the
corresponding properties of \(G\); active-at-threshold values are these left
limits, as in the fixed-discount convention of \cite[Lemma~10]{nmmor20}.
Therefore Assumption~\ref{ass:regFG} holds; Assumption
\ref{ass:avg-limits}\textup{(iv)} gives Assumption~\ref{ass:regfg}.

For the marginal identity, use \eqref{eq:disc-marginal-ibp}. If \(z<x\), then
\[
f_{\beta,x}(z)-m_\beta(z)g_{\beta,x}(z)
=
\int_{(z,x)}g_{\beta,x}(y)\,\mathrm dm_\beta(y).
\]
Passing to the limit, using the marginal-metric convergence derived above, the
locally uniform convergence \(m_\beta\to m\), and \eqref{eq:vd-marginal-pos},
yields
\[
f_x(z)-m(z)g_x(z)
=
\int_{(z,x)}g_x(y)\,\mathrm dm(y).
\]
The case \(x\le z\) follows from \eqref{eq:vd-marginal-neg} and gives
\[
f_x(z)-m(z)g_x(z)
=
-\int_{[x,z]}g_x(y)\,\mathrm dm(y).
\]
Thus \textup{(PCLI3b)} holds.

The average metrics satisfy \textup{(PCLI1)}, \textup{(PCLI2)},
\textup{(PCLI3a)}, and \textup{(PCLI3b)}. The threshold-indexability result, the
identification of the Whittle index with \(m\), and the generalized-inverse
description of the optimal threshold maps follow from Theorem~\ref{the:pcliii}.
\end{proof}

\begin{remark}
\label{re:avg-transfer-summary}
The theorem uses two distinct limit passages. The first is a threshold-metric
LS passage involving the signed resource threshold measure \(\mathrm dG\). Once
\(G\) is c\`adl\`ag and nonincreasing, this passage is usually checked through
convergence of the signed LS measures generated by \((1-\beta)G_{\beta,x}\),
together with locally uniform convergence of \(m_\beta\) to \(m\). The second is
a marginal integration-by-parts passage involving the positive index measures
\(\mathrm dm_\beta\to\mathrm dm\). It does not require \(g_x\) to have locally
bounded variation; this is essential in the average-criterion examples below.
\end{remark}

\section{Application examples}
\label{s:aexamples}
This section applies the average-transfer theorem to three scalar-state models:
two benchmark models for which it recovers known long-run average indices, and
one scalar Kalman-filter model where it isolates the remaining exceptional-state
limits needed for a complete average-cost theorem.

\subsection{Optimal web crawling under the long-run average criterion}
\label{s:owcavg}

We first consider the deterministic web crawling model analyzed directly in
\cite{avraBorkar18} under the long-run average criterion and through discounted
PCL-indexability in \cite[\S~12.1]{nmmor20}. The goal is to recover the
average index from the discounted threshold metrics via
Theorem~\ref{the:avg-pcli-transfer}.

The one-step reward and resource consumption are \(r(x,a)=xa\) and
\(c(x,a)=Ca\), with \(C>0\). The state evolves as
\[
X_{t+1}
=
\begin{cases}
\ell+\alpha X_t, & A_t=0,\\
\ell, & A_t=1,
\end{cases}
\qquad 0<\alpha<1,
\]
where \(\ell=(1-\alpha)b\), \(b>0\), and
\(\Stsp=[\ell,u]\), with \(u=\ell/(1-\alpha)=b\). Hence
\(\Zsp=\mathbb R\); thresholds \(z<\ell\) and \(z\ge u\) represent,
respectively, the all-active and all-passive policies.

Let \(\phi(x)\triangleq \ell+\alpha x\) and
\(\phi_t(x)\triangleq u-(u-x)\alpha^t\), \(t=0,1,\ldots\). For
\(x\in\Stsp\) and \(z\in\mathbb R\), define
\[
\tau(x,z)\triangleq \min\{t\ge1\colon\phi_t(x)>z\},
\]
with \(\tau(x,z)=\infty\) if the set is empty. For \(\ell\le z<u\), write
\(t(z)\triangleq\tau(\ell,z)\), so that
\(\phi_{t(z)-1}(\ell)\le z<\phi_{t(z)}(\ell)\).

\begin{proposition}
\label{pro:owc_avg_acoe}
Assumption~\ref{ass:avg-acoe} holds for the web crawling model.
\end{proposition}

\begin{proof}
Fix \(\lambda\in\mathbb R\). The state and action spaces are compact and
finite, respectively, and \(r_\lambda(x,a)=xa-\lambda Ca\) is bounded and
continuous. The transition kernel is weakly continuous because
\(\kappa_x^0=\delta_{\phi(x)}\), \(\kappa_x^1=\delta_\ell\), and
\(\phi\) is continuous.

For \(\beta\in(0,1)\),
\[
(\mathcal T_{\beta,\lambda}v)(x)
=
\max\{x-\lambda C+\beta v(\ell),\,\beta v(\phi(x))\}.
\]
If \(v\) is \(1\)-Lipschitz, the active branch is \(1\)-Lipschitz and the
passive branch is \(\alpha\beta\)-Lipschitz; hence
\(\mathcal T_{\beta,\lambda}(\mathrm{Lip}_1(\Stsp))\subseteq
\mathrm{Lip}_1(\Stsp)\). Lemma~\ref{lem:avg-lip-relative} gives uniformly
bounded equicontinuous normalized discounted relative reward functions.
Standard compact weak-continuity vanishing-discount results, applied to the
cost \(\tilde c_\lambda(x,a)=\lambda c(x,a)-r(x,a)\), yield the average
optimality equation and optimal-action characterization; see, e.g.,
\citet[Theorem~3.2]{feinbergLiang22}.
\end{proof}

For \(\ell\le z<u\), \cite[\S~12.1]{nmmor20} gives
\[
F_{\beta,\ell}(z)
=
\frac{\beta^{t(z)}\phi_{t(z)}(\ell)}{1-\beta^{t(z)+1}},
\qquad
G_{\beta,\ell}(z)
=
\frac{\beta^{t(z)}C}{1-\beta^{t(z)+1}},
\]
and, for arbitrary \(x\in\Stsp\),
\[
F_{\beta,x}(z)
=
\begin{cases}
x+\beta F_{\beta,\ell}(z), & x>z,\\
\beta^{\tau(x,z)}
\bigl(\phi_{\tau(x,z)}(x)+\beta F_{\beta,\ell}(z)\bigr), & x\le z,
\end{cases}
\]
\[
G_{\beta,x}(z)
=
\begin{cases}
C+\beta G_{\beta,\ell}(z), & x>z,\\
\beta^{\tau(x,z)}
\bigl(C+\beta G_{\beta,\ell}(z)\bigr), & x\le z.
\end{cases}
\]
For \(z<\ell\), the policy is all-active, so
\(F_{\beta,x}(z)=x+\beta\ell/(1-\beta)\) and
\(G_{\beta,x}(z)=C/(1-\beta)\). For \(z\ge u\), it is all-passive, so
\(F_{\beta,x}(z)=G_{\beta,x}(z)=0\). Thus
\begin{equation}
\label{eq:owc_avg_FG}
F(z)=
\begin{cases}
\ell, & z<\ell,\\
\phi_{t(z)}(\ell)/(t(z)+1), & \ell\le z<u,\\
0, & z\ge u,
\end{cases}
\qquad
G(z)=
\begin{cases}
C, & z<\ell,\\
C/(t(z)+1), & \ell\le z<u,\\
0, & z\ge u.
\end{cases}
\end{equation}
The bias limits exist because the discounted metrics are rational in \(\beta\)
with at most a simple pole at \(1\); transition compatibility is immediate from
the next states \(\ell\) and \(\phi(x)\).

The discounted marginal metrics are
\[
f_{\beta,x}(z)=x+\beta F_{\beta,\ell}(z)-\beta F_{\beta,\phi(x)}(z),
\qquad
g_{\beta,x}(z)=C+\beta G_{\beta,\ell}(z)-\beta G_{\beta,\phi(x)}(z).
\]
Letting \(\beta\uparrow1\), for \(\ell\le z<u\), gives
\begin{equation}
\label{eq:owc_avg_fg}
f_x(z)
=
x-\phi_{\tau(x,z)}(x)+\tau(x,z)F(z),
\qquad
g_x(z)
=
\tau(x,z)G(z)
=
\frac{\tau(x,z)}{t(z)+1}C .
\end{equation}
For boundary-equivalent thresholds, \(f_x(z)=(1-\alpha)x\) and \(g_x(z)=C\)
when \(z<\ell\), while \(f_x(z)=x\) and \(g_x(z)=C\) when \(z\ge u\).
Since \(\tau(x,x)=1\) for \(\ell\le x<u\), the average MP index is
\begin{equation}
\label{eq:owc_avg_index}
m(x)
=
\frac{(t+1)(x-\phi(x))+\phi_t(\ell)}{C},
\qquad
\phi_{t-1}(\ell)\le x<\phi_t(\ell),\quad t=1,2,\ldots,
\qquad
m(u)=u/C .
\end{equation}
The endpoint value is the limit as \(x\uparrow u\). With
\(D_x(z)\triangleq f_x(z)-m(z)g_x(z)\), the identity
\(m(z)G(z)=z-\phi(z)+F(z)\) yields, for \(\ell\le z<u\),
\begin{equation}
\label{eq:owc_avg_D_identity}
D_x(z)
=
x-\phi_{\tau(x,z)}(x)+\tau(x,z)(\phi(z)-z)
=
(1-\alpha)
\sum_{i=0}^{\tau(x,z)-1}\bigl(\phi_i(x)-z\bigr).
\end{equation}
For \(z<\ell\), \(D_x(z)=(1-\alpha)(x-\ell)\); for \(z\ge u\),
\(D_x(z)=x-u\).

\begin{proposition}
\label{pro:owc_avg_limits}
For the web crawling model, Assumptions~\ref{ass:disc-pcli} and
\ref{ass:avg-limits} hold with \(\Exsp=\varnothing\).
\end{proposition}

\begin{proof}
Assumption~\ref{ass:disc-pcli} follows from \cite[Proposition~8]{nmmor20}.
The displayed formulas give the required gain, bias, and marginal limits. For
\(\ell\le z<u\), \(g_x(z)=\tau(x,z)C/(t(z)+1)>0\), and the boundary formulas
give \(g_x(z)=C>0\); hence \(\Exsp=\varnothing\).

The index in \eqref{eq:owc_avg_index} is affine with slope
\((t+1)(1-\alpha)/C>0\) on each interval
\([\phi_{t-1}(\ell),\phi_t(\ell))\), with matching one-sided values and
\(m(u)=u/C\). Thus \(m\) is continuous and strictly increasing. The discounted
indices are continuous, nondecreasing, and piecewise affine with the same
breakpoints, and their coefficients converge interval by interval to those in
\eqref{eq:owc_avg_index}. P\'olya's uniform-convergence theorem gives uniform
convergence on \(\Stsp\), hence locally uniform convergence after endpoint
constant continuation \cite[Problems~II.126--II.127, pp.~63,
225]{polyaSzego98}.

The functions \(F\) and \(G\) in \eqref{eq:owc_avg_FG} are c\`adl\`ag step
functions of locally bounded variation; the only jump accumulation is at
\(u\), with summable jump tails. Moreover, \(G\) is nonincreasing and
nonconstant. For each fixed \(x\), \(f_x\) and \(g_x\) are finite-valued
c\`adl\`ag step functions.

For the threshold LS passage, the discounted PCL identity gives
\[
(1-\beta)\int_{(z_1,z_2]}m_\beta(y)\,\mathrm dG_{\beta,x}(y)
=
(1-\beta)\bigl(F_{\beta,x}(z_2)-F_{\beta,x}(z_1)\bigr),
\]
whose limit is \(F(z_2)-F(z_1)\). Since \(F\) and \(G\) are step functions with
summable jumps, it suffices to check atom slopes. At \(z=\phi_t(\ell)\),
\[
m(z)
=
\frac{F(z^{\scriptscriptstyle -})-F(z)}
     {G(z^{\scriptscriptstyle -})-G(z)}
=
\frac{(t+2)\phi_t(\ell)-(t+1)\phi_{t+1}(\ell)}{C},
\]
which agrees with \eqref{eq:owc_avg_index}. Hence \eqref{eq:vd-gain} and
\textup{(PCLI3a)} hold.

For the marginal passage, \eqref{eq:owc_avg_D_identity} gives, on intervals
where \(t(\cdot)\) and \(\tau(x,\cdot)\) are constant,
\[
m'(z)=\frac{(t(z)+1)(1-\alpha)}{C},
\qquad
g_x(z)m'(z)=\tau(x,z)(1-\alpha),
\qquad
D_x'(z)=-g_x(z)m'(z).
\]
Since \(m\) and \(D_x\) are continuous at the breakpoints, integration gives
\[
D_x(z)
=
\begin{cases}
\displaystyle
\int_{(z,x)}g_x(y)\,\mathrm dm(y), & z<x,\\[2.5ex]
\displaystyle
-\int_{[x,z]}g_x(y)\,\mathrm dm(y), & x\le z.
\end{cases}
\]
The boundary cases follow from constant continuation of \(m\) outside
\([\ell,u]\). Combining this identity with the discounted marginal
integration-by-parts identity gives \eqref{eq:vd-marginal-pos}--
\eqref{eq:vd-marginal-neg} and \textup{(PCLI3b)}.
\end{proof}

\begin{proposition}
\label{pro:owc_avg}
The web crawling project is threshold-indexable under the long-run average
criterion. Its Whittle index is the average MP index \(m\) in
\eqref{eq:owc_avg_index}, and its optimal threshold maps are the generalized
inverses of \(m\).
\end{proposition}

\begin{proof}
Propositions~\ref{pro:owc_avg_acoe} and~\ref{pro:owc_avg_limits} verify
Assumptions~\ref{ass:avg-acoe}, \ref{ass:disc-pcli}, and~\ref{ass:avg-limits};
the claim follows from Theorem~\ref{the:avg-pcli-transfer}.
\end{proof}

\begin{remark}
\label{re:owc_avg}
The net reward \(xa-\lambda Ca\) is equivalent, after adding \(C\lambda\) per
period, to the passive-subsidy form \(xa+C\lambda(1-a)\). Thus the subsidy-unit
index is \(Cm(x)\), while the present price-per-unit-resource index is
\(m(x)\). The proof route differs from \cite{avraBorkar18}: it transfers
discounted PCL metric identities rather than deriving threshold optimality and
indexability directly from the average-reward bias function.
\end{remark}

\subsection{Optimal dynamic transmission over a noisy channel}
\label{s:ocam}

We next consider the dynamic transmission model of \cite{liuZhao10}. Its
discounted PCL analysis was completed in \cite[\S~12.2]{nmmor20}; here we
recover the long-run average index, originally obtained in \cite{liuZhao10} by
a dynamic-programming route, through metric transfer.

The channel is good \((1)\) or bad \((0)\), with transition probabilities
\(p\) for \(1\to0\) and \(q\) for \(0\to1\). Assume \(p>0\), \(q>0\), and
\(\rho\triangleq1-p-q>0\). The state is the belief
\(X_t\in\Stsp\triangleq[0,1]\). If \(A_t=1\), then
\(X_{t+1}=q+\rho\) with probability \(X_t\), and \(X_{t+1}=q\) otherwise. If
\(A_t=0\), then \(X_{t+1}=\phi(X_t)\), where
\(\phi(x)\triangleq q+\rho x\). The reward and resource functions are
\(r(x,a)=ax\) and \(c(x,a)=a\), and \(\Zsp=\mathbb R\).

\begin{proposition}
\label{pro:ocam_avg_acoe}
Assumption~\ref{ass:avg-acoe} holds for the noisy-channel model.
\end{proposition}

\begin{proof}
For fixed \(\lambda\), \(r_\lambda(x,a)=a(x-\lambda)\) is bounded and
continuous, and the kernel is weakly continuous because
\[
        \kappa_x^0=\delta_{\phi(x)},
        \qquad
        \kappa_x^1=x\,\delta_{q+\rho}+(1-x)\,\delta_q .
\]
The discounted Bellman operator is
\[
(\mathcal T_{\beta,\lambda}v)(x)
=
\max\Bigl\{
x-\lambda+\beta[xv(q+\rho)+(1-x)v(q)],\,
\beta v(\phi(x))
\Bigr\}.
\]
Set \(L\triangleq1/(1-\rho)=1/(p+q)\). If \(v\) is \(L\)-Lipschitz, the
passive branch is \(\beta\rho L\)-Lipschitz and the active branch has
Lipschitz constant at most
\(1+\beta|v(q+\rho)-v(q)|\le1+\beta\rho L\le L\). Hence
\(\mathcal T_{\beta,\lambda}(\mathrm{Lip}_L(\Stsp))\subseteq
\mathrm{Lip}_L(\Stsp)\). Lemma~\ref{lem:avg-lip-relative} and the compact
weak-continuity vanishing-discount argument used above yield
Assumption~\ref{ass:avg-acoe}.
\end{proof}

The passive update has fixed point
\(\phi_\infty\triangleq q/(1-\rho)=q/(p+q)\), with iterates
\[
        \phi_t(x)=\phi_\infty-(\phi_\infty-x)\rho^t,
        \qquad
        \phi_{-t}(z)=\phi_\infty-(\phi_\infty-z)\rho^{-t} .
\]
For fixed \(x\), the threshold partition is generated by
\[
\mathcal D(x)
\triangleq
\{0,1,\phi_\infty\}
\cup
\{\phi_t(x):t\ge0\}
\cup
\{\phi_t(q):t\ge0\}
\cup
\{\phi_t(q+\rho):t\ge0\},
\]
with possible accumulation only at \(\phi_\infty\). Following
\cite[\S~12.2]{nmmor20}, the four regimes are
\[
\textup{I: }z<q,
\qquad
\textup{II: }q\le z<\phi_\infty,
\qquad
\textup{III: }\phi_\infty\le z<q+\rho,
\qquad
\textup{IV: }z\ge q+\rho .
\]

\begin{proposition}
\label{pro:ocam_avg_limits}
For the noisy-channel model, Assumptions~\ref{ass:disc-pcli} and
\ref{ass:avg-limits} hold with \(\Exsp=\varnothing\).
\end{proposition}

\begin{proof}
Assumption~\ref{ass:disc-pcli} follows from \cite[\S~12.2]{nmmor20}. On each
interval of the partition generated by \(\mathcal D(x)\), the first-crossing
indices are fixed; the discounted threshold and marginal metrics are rational
functions of \(\beta\), with at most a simple pole at \(1\) in the threshold
metrics. Coefficientwise limits give state-independent gains \(F,G\), finite
bias metrics \(h_F,h_G\), and transition compatibility at the next states
\(q+\rho\), \(q\), and \(\phi(x)\). The marginal limits follow from
\[
\begin{split}
f_{\beta,x}(z)
&=
x+\beta\bigl[xF_{\beta,q+\rho}(z)+(1-x)F_{\beta,q}(z)\bigr]
-\beta F_{\beta,\phi(x)}(z),\\
g_{\beta,x}(z)
&=
1+\beta\bigl[xG_{\beta,q+\rho}(z)+(1-x)G_{\beta,q}(z)\bigr]
-\beta G_{\beta,\phi(x)}(z).
\end{split}
\]

We verify limiting marginal-resource positivity. In regimes I and IV,
\(g_x(z)=1\). In regime III,
\[
g_x(z)
=
\frac{p+x}{p}\,1_{\{\phi(x)\le z\}}
+
\frac{(p+q)x-q}{p}\,1_{\{\phi(x)>z\}} ,
\]
which is positive because \(\phi(x)>z\ge\phi_\infty\) implies
\(x>\phi_\infty=q/(p+q)\). In regime II, let \(t=t(z)\ge1\) satisfy
\(\phi_{t-1}(q)\le z<\phi_t(q)\), set \(\xi_t\triangleq\phi_t(q)\),
\(d_t\triangleq p(t+1)+\xi_t\), and let \(s=s(x,z)\ge1\) be the first crossing
index. The limiting formula is
\[
        g_x(z)=\frac{s(p+\xi_t)+t\{x-\phi_s(x)\}}{d_t}.
\]
The denominator is positive. If \(x\ge\phi_\infty\), then \(s=1\) and the
numerator is positive. If \(x<\phi_\infty\), the numerator is affine increasing
in \(x\) on each fixed-\(s\) interval. For \(1\le s\le t\), the left-boundary
value is increasing in \(z\) and is bounded below by its value at
\(z=\phi_{t-1}(q)\):
\[
sp
+
q\left[
 s\sum_{j=0}^{t}\rho^j
 -t\sum_{j=t-s}^{t-1}\rho^j
\right]
>0,
\]
because \(0<\rho<1\). For \(s=t+1\), the left boundary is \(x=0\), and the
numerator is \((t+1)p+\xi_t>0\). Hence \(g_x(z)>0\) in all regimes.

The average MP index is \(m(x)=\lim_{\beta\uparrow1}m_\beta(x)\). In the outer
regimes, \(m(x)=x\) for \(x<q\), \(m(x)=x/(p+x)\) for
\(\phi_\infty\le x<q+\rho\), and \(m(x)=x\) for \(x\ge q+\rho\); in regime II,
\(m\) is the coefficientwise limit of the formula in \cite[\S~12.2]{nmmor20}.
The discounted indices are continuous and nondecreasing, with matching
one-sided limits at breakpoints. P\'olya's theorem gives uniform convergence on
\([0,1]\) \cite[Problems~II.126--II.127, pp.~63, 225]{polyaSzego98}, and the
marginal convergence plus \(g_x(x)>0\) gives \(m(x)=f_x(x)/g_x(x)\).

The gain metrics \(F\) and \(G\) are c\`adl\`ag step functions. Away from
\(\phi_\infty\), only finitely many partition points occur on compact
intervals; near \(\phi_\infty\), regenerative formulas give summable jump tails
for \(F\) and \(G\) on both sides. Thus \(F,G\) are locally of bounded variation
and \(G\) is nonincreasing and nonconstant. For fixed \(x\), \(f_x,g_x\) are
finite-valued c\`adl\`ag step functions with jumps in \(\mathcal D(x)\), though
they need not have bounded variation.

The threshold LS identity follows by atom-by-atom convergence. At each atom
\(z\) with \(G(z^{\scriptscriptstyle -})-G(z)>0\),
\[
        m(z)
        =
        \frac{F(z^{\scriptscriptstyle -})-F(z)}
             {G(z^{\scriptscriptstyle -})-G(z)} .
\]
Together with the summable gain-tail estimates, this gives \(\mathrm dF(z)=
m(z)\,\mathrm dG(z)\) and \eqref{eq:vd-gain}.

For \textup{(PCLI3b)}, the discounted identity gives
\[
f_{\beta,x}(z)-m_\beta(z)g_{\beta,x}(z)
=
\begin{cases}
\displaystyle
\int_{(z,x)}g_{\beta,x}(y)\,\mathrm dm_\beta(y), & z<x,\\[2.5ex]
\displaystyle
-\int_{[x,z]}g_{\beta,x}(y)\,\mathrm dm_\beta(y), & x\le z.
\end{cases}
\]
The left-hand side converges from the explicit formulas. On the right,
termwise convergence holds away from \(\phi_\infty\), and itinerary-tail
estimates control the accumulating intervals near \(\phi_\infty\). Since
\(m_\beta\to m\) locally uniformly and the \(m_\beta\)'s are continuous
nondecreasing, \(\mathrm dm_\beta\Rightarrow\mathrm dm\). The integrand
\(\1_Ig_x\) is bounded and \(\mathrm dm\)-a.e. continuous, so the marginal LS
passages \eqref{eq:vd-marginal-pos}--\eqref{eq:vd-marginal-neg} follow.
\end{proof}

\begin{remark}
\label{re:ocam_avg_accumulation}
The point \(\phi_\infty\) is an interior accumulation point of threshold
partition points, with \(\phi_t(q)\uparrow\phi_\infty\) and
\(\phi_t(q+\rho)\downarrow\phi_\infty\). The orbit of the initial state may add
another interlacing jump sequence, so \(z\mapsto g_x(z)\) can fail to have
locally bounded variation there. This is harmless because the average marginal
identity uses \(g_x\) only as an integrand with respect to \(\mathrm dm\).
\end{remark}

\begin{proposition}
\label{pro:ocam_avg}
The noisy-channel transmission model is threshold-indexable under the long-run
average criterion. Its Whittle index is \(m=\lim_{\beta\uparrow1}m_\beta\), and
its optimal threshold maps are the generalized inverses of \(m\).
\end{proposition}

\begin{proof}
Propositions~\ref{pro:ocam_avg_acoe} and~\ref{pro:ocam_avg_limits} verify
Assumptions~\ref{ass:avg-acoe}, \ref{ass:disc-pcli}, and~\ref{ass:avg-limits};
the result follows from Theorem~\ref{the:avg-pcli-transfer}.
\end{proof}

\begin{remark}
\label{re:ocam_avg}
The transferred average MP index is the long-run average Whittle index of
\cite{liuZhao10}. Since \(c(x,a)=a\), the price-per-unit-resource and
passive-subsidy normalizations differ only by adding \(\lambda\) per period:
\[
        xa-\lambda a
        \quad\longleftrightarrow\quad
        xa+\lambda(1-a),
\]
so the numerical index is unchanged.
\end{remark}

\subsection{Scalar Kalman-filter bandits: regular average transfer and 
exceptional-state extension}
\label{s:kalmanavg}

We finally consider the scalar Kalman-filter restless bandit. The target-tracking
restless bandit formulation was introduced in \cite{lascalaMoran06}; the
scalar discrete-time PCL approach was introduced in \cite{nmsv09}; and the
discounted PCL-indexability analysis was completed by \citet{danceSi19}. We do
not claim a complete long-run average threshold-indexability theorem. Instead,
we prove the regular periodic-itinerary part of the transfer and identify the
exceptional-state metric limits needed for the full average-cost result.

This conditional treatment is natural because the discounted proof in
\cite{danceSi19} is already highly nontrivial: it uses maps with gaps, 
Christoffel and Sturmian words, M\"obius transformations, and majorization
inequalities. For background on Christoffel and Sturmian words, see
\cite{bersteletal09}. The purpose here is to specify what remains to be
proved in the average limit, not to redo the symbolic analysis.

We use the variance-cost case \(C(x)=x\), with the standard parameters in
\cite{danceSi19},
\[
        r_{\rm DS}=1,
        \qquad
        \theta_0=0,
        \qquad
        \theta_1>0 .
\]
The state is \(X_t\in\Stsp\triangleq[0,\infty)\), and the deterministic update
maps are
\[
        \phi_0(x)=x+1,
        \qquad
        \phi_1(x)=\frac{x+1}{\theta_1(x+1)+1}.
\]
In the reward notation used here, \(r(x,a)=-x\) and \(c(x,a)=a\), so the
\(\lambda\)-price reward is \(-x-\lambda a\), equivalently the one-stage cost
is \(x+\lambda a\). Since the all-passive trajectory is \(X_t=x+t\), its
average variance cost diverges; hence \(\infty\notin\Zsp\) and
\(\Zsp=\mathbb R\). Thresholds \(z<0\) represent the all-active policy.

The symbolic results of \cite{danceSi19} apply to the original maps
\(\phi_0,\phi_1\). Although \(\phi_0(x)=x+1\) is not contractive, the increasing
change of variables \(x\mapsto x/(x+1)\) transforms the dynamics into the
contractive setting used there. Lemma~\ref{lem:avg-lip-relative} does not apply
on the half-line: in reward form,
\[
(\mathcal T_{\beta,\lambda}v)(x)
=
\max\{-x-\lambda+\beta v(\phi_1(x)),\,-x+\beta v(\phi_0(x))\},
\]
and the passive branch does not preserve a \(\beta\)-uniform Lipschitz class.
Thus Assumption~\ref{ass:avg-acoe} should be verified separately using
noncompact unbounded-cost theory, e.g., \cite{feinbergetal12,feinbergLiang22}.

For a finite word \(w=w_1\cdots w_n\), write
\(\phi_w\triangleq \phi_{w_n}\circ\cdots\circ\phi_{w_1}\), and let \(y_w\) be
the fixed point of \(\phi_w\). Denote by \(|w|\) the length of \(w\) and by \(|w|_1\) the number of ones in it. 
Also, \(w_{k:l}\) denotes the subword \(w_k\cdots w_l\).
Let \(y_1\) be the fixed point of \(\phi_1\). For
each Christoffel word \(0p1\), define
\[
        n_p\triangleq |0p1|,
        \qquad
        q_p\triangleq |0p1|_1,
        \qquad
        I_p\triangleq [y_{01p},\,y_{10p}),
\]
and
\[
        w_p^{(k)}
        \triangleq
        (0p1)_{k:n_p}(0p1)_{1:k-1},
        \qquad
        x_k^p\triangleq y_{w_p^{(k)}},
        \qquad k=1,\ldots,n_p,
\]
with \(\bar x_p\triangleq n_p^{-1}\sum_{k=1}^{n_p}x_k^p\). The relevant state values split into the regular set
\[
\mathcal R_{\rm K}
\triangleq
[0,y_1)
\cup
\bigcup_{0p1\textup{ Christoffel}} I_p
\]
and the exceptional set
\[
{\mathcal E}_{\rm K}
\triangleq
\{y_1\}
\cup
\{y_{10p}\colon0p1\textup{ Christoffel}\}
\cup
\{y_s\colon0s\textup{ Sturmian \(M\)-word}\}.
\]
For every regular finite threshold value, namely \(z<y_1\) or
\(z\in I_p\) for some Christoffel word \(0p1\), the symbolic itinerary of the
\(z\)-threshold trajectory, from any initial state \(x\), is eventually
periodic. In the case \(z<y_1\), the itinerary is eventually all active; in
the case \(z\in I_p\), it is eventually periodic with primitive period word
\(0p1\), up to cyclic shift. At state values in \({\mathcal E}_{\rm K}\), the average
diagonal marginal resource \(g_x(x)\) vanishes, so the MP index cannot be
defined by the diagonal ratio \(f_x(x)/g_x(x)\) and must instead be supplied by
continuous extension. 

Here and below, an \emph{Abelian limit} means a limit obtained from discounted weighted sums as
\(\beta\uparrow1\), typically of the form
\((1-\beta)\sum_{n\ge0}\beta^n a_n\) or a normalized variant.

Before checking the PCL-indexability conditions, we record the regular average
metric limits. These identify the threshold-policy gain and bias metrics
on the regular threshold regions and, through the average marginal formulas,
the corresponding marginal metrics and regular diagonal MP ratio. They are
used in all three subsequent checks: \textup{(PCLI1)}, \textup{(PCLI2)}, and
\textup{(PCLI3)}.

After this preparatory step, the subsection is organized around the three
PCL-indexability conditions. This organization separates what can be proved on
regular periodic-itinerary regions from what remains to be established at
exceptional symbolic-itinerary states. Thus the full long-run average
indexability problem is reduced to three explicit metric-limit questions:
average marginal-resource positivity, continuous extension of the
regular MP index, and completion of the singular PCL identities.

\subsubsection{Regular average metric limits}
\label{ss:kf-regular-metrics}

We first identify the average gain and bias metrics on the regular
periodic-itinerary threshold regions. These metrics generate the average
marginal metrics and the regular diagonal MP ratio used in the subsequent
\textup{(PCLI1)}--\textup{(PCLI3)} discussion.
Let \((X_t,A_t)\) denote the state-action trajectory of the \(z\)-threshold policy started
from \(X_0=x\).

If \(z<y_1\), the trajectory is eventually always active. Define
\[
        T_1(x,z)
        \triangleq
        \min\{T\ge0: A_t=1 \textup{ for all } t\ge T\}.
\]
For \(t\ge T_1(x,z)\), the state evolves under repeated application of
\(\phi_1\). Since \(y_1\) is the attracting fixed point of \(\phi_1\),
\(X_t-y_1\) converges to zero geometrically. Define
\begin{equation}
\label{eq:kf_bias_all_active}
        B_F^{1}(x,z)
        \triangleq
        \sum_{t\ge0}(-X_t+y_1),
        \qquad
        B_G^{1}(x,z)
        \triangleq
        \sum_{t\ge0}(A_t-1).
\end{equation}
The first series is absolutely convergent, and the second is finite.

If \(z\in I_p\), then, by \cite[Theorem~12 and Corollary~13]{danceSi19},
after a finite transient the symbolic itinerary is periodic with primitive
period word \(0p1\), up to cyclic shift. Define \(T_p(x,z)\) to be the least
\(T\ge0\) for which there exists a cyclic shift
\[
        a^p=a_0^p a_1^p\cdots a_{n_p-1}^p
\]
of the word \(0p1\) such that, with periodic extension
\(a_{s+n_p}^p=a_s^p\),
\[
        A_{T+s}=a_s^p,\qquad s\ge0 .
\]
Then
\[
        \sum_{s=0}^{n_p-1}a_s^p=q_p .
\]

The corresponding limiting periodic orbit is defined from this same cyclic
phase. For \(s=0,\ldots,n_p-1\), let
\[
        w_{p,s}
        \triangleq
        a_s^p a_{s+1}^p\cdots a_{n_p-1}^p a_0^p\cdots a_{s-1}^p
\]
be the cyclic shift of \(a^p\) starting at phase \(s\), and define
\[
        \xi_s^p\triangleq y_{w_{p,s}} .
\]
Extend \((\xi_s^p)\) periodically by \(\xi_{s+n_p}^p=\xi_s^p\). Then
\[
        \xi_{s+1}^p=\phi_{a_s^p}(\xi_s^p),
        \qquad s\ge0,
\]
with indices read modulo \(n_p\). Thus
\((\xi_0^p,\ldots,\xi_{n_p-1}^p)\) is just a cyclic relabeling of the orbit
\((x_1^p,\ldots,x_{n_p}^p)\) defined above, and hence
\[
        \bar x_p=\frac1{n_p}\sum_{s=0}^{n_p-1}\xi_s^p .
\]

Putting \(T=T_p(x,z)\), define the trajectory error by
\[
        e_s\triangleq X_{T+s}-\xi_s^p,\qquad s\ge0 .
\]
Convergence to the limiting periodic orbit gives constants \(C<\infty\) and
\(\eta\in(0,1)\) such that
\[
        |e_s|\le C\eta^s,\qquad s\ge0 .
\]
The phase of \((a_s^p)\), \((\xi_s^p)\), and \((e_s)\) depends on \(x\) and
\(z\), but this dependence is suppressed in the notation.

Define the centered periodic resource and reward sequences
\begin{equation}
\label{eq:kf_centered_periodic_sequences}
        d_s^G\triangleq a_s^p-\frac{q_p}{n_p},
        \qquad
        d_s^F\triangleq -\xi_s^p+\bar x_p,
        \qquad s=0,\ldots,n_p-1,
\end{equation}
and extend them periodically. Thus
\[
        \sum_{s=0}^{n_p-1}d_s^G
        =
        \sum_{s=0}^{n_p-1}d_s^F
        =
        0 .
\]

Finally set
\begin{equation}
\label{eq:kf_periodic_bias_G}
        B_G^p(x,z)
        \triangleq
        \sum_{t=0}^{T-1}
        \left(A_t-\frac{q_p}{n_p}\right)
        -
        \frac1{n_p}\sum_{s=0}^{n_p-1}s\,d_s^G
\end{equation}
and
\begin{equation}
\label{eq:kf_periodic_bias_F}
        B_F^p(x,z)
        \triangleq
        \sum_{t=0}^{T-1}(-X_t+\bar x_p)
        -
        \frac1{n_p}\sum_{s=0}^{n_p-1}s\,d_s^F
        -
        \sum_{s\ge0}e_s .
\end{equation}
The sums in \eqref{eq:kf_periodic_bias_G} are finite. In
\eqref{eq:kf_periodic_bias_F}, the only infinite series is
\(\sum_{s\ge0}e_s\), and it is absolutely convergent because
\(|e_s|\le C\eta^s\) with \(0<\eta<1\). This guarantees that
\(B_F^p(x,z)\) is finite and will justify the dominated-convergence passage
\(\sum_{s\ge0}\beta^s e_s\to\sum_{s\ge0}e_s\).

\begin{proposition}[Regular average gain and bias metrics]
\label{pro:kf_avg_gains}
For regular threshold values, the gain and bias limits in
Assumption~\ref{ass:avg-limits}\textup{(i)} hold as follows.

\begin{enumerate}[label=\textup{(\alph*)},leftmargin=*]

\item
If \(z<y_1\), then, for every initial state \(x\),
\[
        F(z)=-y_1,
        \qquad
        G(z)=1,
        \qquad
        h_F(x,z)=B_F^{1}(x,z),
        \qquad
        h_G(x,z)=B_G^{1}(x,z).
\]

\item
If \(z\in I_p\), then, for every initial state \(x\),
\[
        F(z)=-\bar x_p,
        \qquad
        G(z)=q_p/n_p,
        \qquad
        h_F(x,z)=B_F^p(x,z),
        \qquad
        h_G(x,z)=B_G^p(x,z).
\]

\end{enumerate}
Equivalently, in both cases,
\[
        (1-\beta)F_{\beta,x}(z)\to F(z),
        \qquad
        (1-\beta)G_{\beta,x}(z)\to G(z),
\]
and
\[
        F_{\beta,x}(z)-\frac{F(z)}{1-\beta}\to h_F(x,z),
        \qquad
        G_{\beta,x}(z)-\frac{G(z)}{1-\beta}\to h_G(x,z).
\]
\end{proposition}

\begin{proof}
First suppose \(z<y_1\). By the definition of \(T_1(x,z)\), \(A_t=1\) for
all \(t\ge T_1(x,z)\), and \(X_t-y_1\) is geometrically summable. Hence
\[
\begin{split}
F_{\beta,x}(z)
&=
\sum_{t\ge0}\beta^t(-X_t)  \\
&=
-\frac{y_1}{1-\beta}
+
\sum_{t\ge0}\beta^t(-X_t+y_1),
\end{split}
\]
where the second sum has a finite limit as \(\beta\uparrow1\), namely
\(B_F^{1}(x,z)\). Thus
\[
        (1-\beta)F_{\beta,x}(z)\to -y_1,
        \qquad
        F_{\beta,x}(z)-\frac{-y_1}{1-\beta}
        \to B_F^{1}(x,z).
\]
Similarly,
\[
G_{\beta,x}(z)
=
\sum_{t\ge0}\beta^t A_t
=
\frac{1}{1-\beta}
+
\sum_{t\ge0}\beta^t(A_t-1),
\]
and the second sum is finite because \(A_t=1\) eventually. Therefore
\[
        (1-\beta)G_{\beta,x}(z)\to 1,
        \qquad
        G_{\beta,x}(z)-\frac{1}{1-\beta}
        \to B_G^{1}(x,z).
\]
This proves part \textup{(a)}.

Now suppose \(z\in I_p\), and put \(T=T_p(x,z)\). The periodic action tail
gives
\[
\begin{split}
G_{\beta,x}(z)-\frac{q_p/n_p}{1-\beta}
&=
\sum_{t=0}^{T-1}\beta^t
\left(A_t-\frac{q_p}{n_p}\right)
+
\beta^T\sum_{s\ge0}\beta^s d_s^G .
\end{split}
\]
The finite initial term converges to the first term in
\eqref{eq:kf_periodic_bias_G}. For the periodic tail,
\[
        \sum_{s\ge0}\beta^s d_s^G
        =
        \frac{\sum_{s=0}^{n_p-1}\beta^s d_s^G}
             {1-\beta^{n_p}} .
\]
Since \(d^G\) has zero mean over one period,
\[
        \sum_{s=0}^{n_p-1}\beta^s d_s^G
        =
        (\beta-1)\sum_{s=0}^{n_p-1}s\,d_s^G+o(\beta-1),
\]
whereas
\[
        1-\beta^{n_p}
        =
        -n_p(\beta-1)+o(\beta-1).
\]
Therefore
\[
        \sum_{s\ge0}\beta^s d_s^G
        \longrightarrow
        -
        \frac1{n_p}\sum_{s=0}^{n_p-1}s\,d_s^G .
\]
Since \(\beta^T\to1\), it follows that
\[
        G_{\beta,x}(z)-\frac{q_p/n_p}{1-\beta}
        \to
        B_G^p(x,z),
\]
and hence \((1-\beta)G_{\beta,x}(z)\to q_p/n_p\).

For the reward metric, using \(X_{T+s}=\xi_s^p+e_s\),
\[
\begin{split}
F_{\beta,x}(z)-\frac{-\bar x_p}{1-\beta}
&=
\sum_{t=0}^{T-1}\beta^t(-X_t+\bar x_p)
+
\beta^T\sum_{s\ge0}\beta^s(-X_{T+s}+\bar x_p)  \\
&=
\sum_{t=0}^{T-1}\beta^t(-X_t+\bar x_p)
+
\beta^T\sum_{s\ge0}\beta^s d_s^F
-
\beta^T\sum_{s\ge0}\beta^s e_s .
\end{split}
\]
The finite initial term converges to its value at \(\beta=1\). The error term
converges by dominated convergence, since \(\sum_{s\ge0}|e_s|<\infty\):
\[
        \sum_{s\ge0}\beta^s e_s
        \longrightarrow
        \sum_{s\ge0}e_s .
\]
The periodic zero-mean term is handled as above:
\[
        \sum_{s\ge0}\beta^s d_s^F
        =
        \frac{\sum_{s=0}^{n_p-1}\beta^s d_s^F}
             {1-\beta^{n_p}}
        \longrightarrow
        -
        \frac1{n_p}\sum_{s=0}^{n_p-1}s\,d_s^F .
\]
Since \(\beta^T\to1\), we obtain
\[
        F_{\beta,x}(z)-\frac{-\bar x_p}{1-\beta}
        \to
        B_F^p(x,z),
\]
and hence \((1-\beta)F_{\beta,x}(z)\to-\bar x_p\). This proves part
\textup{(b)}.
\end{proof}

\begin{proposition}[Regular average marginal metrics and MP ratio]
\label{pro:kf_avg_marginals}
On the regular threshold regions covered by
Proposition~\ref{pro:kf_avg_gains}, the average marginal reward and resource
metrics generated by the bias metrics are
\[
        f_x(z)
        =
        h_F(\phi_1(x),z)-h_F(\phi_0(x),z),
\]
and
\[
        g_x(z)
        =
        1+h_G(\phi_1(x),z)-h_G(\phi_0(x),z).
\]
Consequently, for regular state values \(x\in\mathcal R_{\rm K}\) with
\(g_x(x)>0\), the regular average MP index is
\[
        m_{\rm reg}(x)
        \triangleq
        \frac{f_x(x)}{g_x(x)}
        =
        \frac{h_F(\phi_1(x),x)-h_F(\phi_0(x),x)}
             {1+h_G(\phi_1(x),x)-h_G(\phi_0(x),x)} .
\]
\end{proposition}

\begin{proof}
For the scalar Kalman-filter bandit, \(r(x,a)=-x\) and \(c(x,a)=a\). Hence
\[
        r(x,1)-r(x,0)=0,
        \qquad
        c(x,1)-c(x,0)=1 .
\]
The transition kernels are deterministic: choosing action \(1\) sends \(x\) to
\(\phi_1(x)\), while choosing action \(0\) sends \(x\) to \(\phi_0(x)\).
Therefore the general average marginal formulas
\eqref{eq:avg-f-from-hF}--\eqref{eq:avg-g-from-hG} give the stated expressions for 
\(f_x(z)\)
and
\(g_x(z)\).
The formula for \(m_{\rm reg}(x)\) is the diagonal MP ratio on regular states.
\end{proof}

\subsubsection{\textup{(PCLI1)}: marginal-resource positivity}
\label{ss:kf-pcli1}

The first average PCL condition requires positivity of the limiting marginal
resource metric off the exceptional diagonal. The discounted analysis gives a
strong positivity result for each fixed \(\beta<1\), but the lower bound
degenerates as \(\beta\uparrow1\), so an additional Abelian-mean argument is
needed in the average limit.

For \(x\in[0,\infty)\) and \(z\in\mathbb R\), let
\[
        a(x,z)\triangleq 1\sigma(\phi_1(x)\mid z),
        \qquad
        b(x,z)\triangleq 0\sigma(\phi_0(x)\mid z),
\]
where \(\sigma(\cdot\mid z)\) denotes the threshold itinerary in
\cite{danceSi19}. Define
\[
        N_k(x,z)
        \triangleq
        |a(x,z)_{1:k}|_1-|b(x,z)_{1:k}|_1,
        \qquad k\ge1 .
\]
The proof of \cite[Proposition~18]{danceSi19} gives \(N_k(x,z)\ge0\), and
\begin{equation}
\label{eq:kf-prefix-surplus}
        g_{\beta,x}(z)
        =
        (1-\beta)\sum_{k\ge1}\beta^{k-1}N_k(x,z).
\end{equation}
Thus fixed-discount positivity only gives a lower bound of order \(1-\beta\);
strict average positivity requires an Abelian positivity result.

\begin{lemma}[Diagonal marginal resource limits]
\label{lem:kf_diag_resource}
The following hold.
\begin{enumerate}[label=\textup{(\roman*)},leftmargin=*]
\item If \(x<y_1\), then \(g_{\beta,x}(x)=1\) for every \(\beta\in(0,1)\), and
\(g_x(x)=1\).
\item If \(x\in I_p\) for a Christoffel word \(0p1\), then
\[
        g_{\beta,x}(x)=\frac{1-\beta}{1-\beta^{n_p}},
        \qquad 0<\beta<1,
\]
and \(g_x(x)=1/n_p>0\).
\item If \(x\in {\mathcal E}_{\rm K}\), then \(g_{\beta,x}(x)=1-\beta\), and \(g_x(x)=0\).
\end{enumerate}
\end{lemma}

\begin{proof}
For \(x<y_1\), \cite[Corollary~13]{danceSi19} gives identical active
continuations after the initial action. For \(x\in I_p\), the same corollary
gives \((1p0)^\infty\) and \((0p1)^\infty\); including the initial action, the
resource sequences are \((01p)^\infty\) and \((10p)^\infty\), whose prefix
surplus equals one once per cycle of length \(n_p\). At \(y_1\), at
Christoffel right endpoints \(y_{10p}\), and at Sturmian points \(y_s\), the
same corollary gives the pairs \(1^\infty\) versus \(01^\infty\),
\((1p0)^\infty\) versus \(0p(01p)^\infty\), and \(1s\) versus \(0s\), yielding
\(g_{\beta,x}(x)=1-\beta\).
\end{proof}

\begin{lemma}[Ces\`aro-to-Abel reduction for prefix surplus]
\label{lem:kf-cesaro-abel}
Let \(N_k\), \(k\ge1\), be nonnegative with \(N_k\le k\). If
\(\bar N\triangleq\lim_{T\to\infty}T^{-1}\sum_{k=1}^T N_k\) exists and is
finite, then
\[
        \lim_{\beta\uparrow1}
        (1-\beta)\sum_{k\ge1}\beta^{k-1}N_k
        =
        \bar N .
\]
\end{lemma}

\begin{proof}
Let \(S_T=\sum_{k=1}^T N_k\) and \(A_T=S_T/T\). Summation by parts gives
\[
(1-\beta)\sum_{k\ge1}\beta^{k-1}N_k
=
(1-\beta)^2\sum_{T\ge1}T A_T\beta^{T-1}.
\]
The weights \(w_T(\beta)\triangleq(1-\beta)^2T\beta^{T-1}\) are
nonnegative and satisfy
\[
        \sum_{T\ge1}w_T(\beta)=1 .
\]
Moreover, for every fixed \(M<\infty\),
\[
        \sum_{T=1}^M w_T(\beta)
        \le
        (1-\beta)^2\frac{M(M+1)}2
        \longrightarrow 0,
        \qquad \beta\uparrow1 .
\]
Thus the weights put asymptotically no mass on any fixed finite initial
segment. Since \(A_T\to\bar N\), for every \(\varepsilon>0\) there is
\(M\) such that \(|A_T-\bar N|\le\varepsilon\) for \(T>M\). Hence
\[
\left|
\sum_{T\ge1}w_T(\beta)A_T-\bar N
\right|
\le
\sum_{T=1}^M w_T(\beta)|A_T-\bar N|
+
\varepsilon,
\]
and the first term tends to zero as \(\beta\uparrow1\). Therefore the
weighted averages converge to \(\bar N\).
\end{proof}

\begin{conjecture}[\textup{(PCLI1)} completion for the Kalman-filter bandit]
\label{con:kf-pcli1}
For every \(x\in[0,\infty)\) and \(z\in\mathbb R\) with \(z\ne x\), the
prefix-surplus sequence has a strictly positive Ces\`aro mean:
\[
        \bar N(x,z)
        \triangleq
        \lim_{T\to\infty}\frac1T\sum_{k=1}^T N_k(x,z)
        >0 .
\]
Consequently, \eqref{eq:kf-prefix-surplus} has the positive Abelian limit
\(\bar N(x,z)\), so \(g_x(z)>0\) whenever \(z\ne x\). At exceptional diagonal
states \(x\in {\mathcal E}_{\rm K}\), Lemma~\ref{lem:kf_diag_resource} gives
\(g_x(x)=0\). The corresponding condition \(f_x(x)=0\) follows once the
normalized marginal-reward limit in Conjecture~\ref{con:kf-pcli2} is finite,
since then \(f_{\beta,x}(x)=(1-\beta)m_\beta(x)\to0\).
\end{conjecture}

\subsubsection{\textup{(PCLI2)}: the MP index and its exceptional extension}
\label{ss:kf-pcli2}

\begin{proposition}[Regular diagonal MP-index limits]
\label{pro:kf_avg_diag}
The discounted diagonal MP indices converge locally uniformly on regular
compact threshold regions to the regular average MP ratio \(m_{\rm reg}\)
defined in Proposition~\ref{pro:kf_avg_marginals}. More explicitly:
\begin{enumerate}[label=\textup{(\roman*)},leftmargin=*]
\item On compact subintervals of \([0,y_1)\), \(g_{\beta,x}(x)=1\), and
\(f_{\beta,x}(x)\) converges locally uniformly. Hence \(m_\beta(x)\) converges
to
\[
        m_{\rm reg}(x)
        =
        \sum_{n\ge1}
        \Bigl[
        \phi_1^{\,n-1}(\phi_0(x))
        -
        \phi_1^{\,n}(x)
        \Bigr],
        \qquad 0\le x<y_1 .
\]
\item If \(K\subset I_p\) is compact, then
\[
        g_{\beta,x}(x)=\frac{1-\beta}{1-\beta^{n_p}},
        \qquad x\in K,
\]
so \(g_{\beta,x}(x)\to1/n_p\) uniformly on \(K\).
\item For the same \(K\), there exists a continuous \(f_p\) such that
\(\sup_{x\in K}|f_{\beta,x}(x)-f_p(x)|\to0\). Consequently,
\[
        \sup_{x\in K}|m_\beta(x)-m_{\rm reg}(x)|\to0,
        \qquad
        m_{\rm reg}(x)
        = n_p f_p(x).
\]
\end{enumerate}
\end{proposition}

\begin{proof}
If \(x<y_1\), the active-start and passive-start continuations are eventually
governed by the same active itinerary, and
\[
        f_{\beta,x}(x)
        =
        \sum_{n\ge1}\beta^{n-1}
        \Bigl[
        \phi_1^{\,n-1}(\phi_0(x))
        -
        \phi_1^{\,n}(x)
        \Bigr].
\]
The two trajectories converge geometrically to \(y_1\), uniformly on compact
subintervals. For \(x\in I_p\), Lemma~\ref{lem:kf_diag_resource} gives the
resource formula. The marginal reward can be written as
\[
f_{\beta,x}(x)
=
\sum_{n\ge1}\beta^{n-1}
\Bigl(
\phi_{((01p)^\infty)_{1:n}}(x)
-
\phi_{((10p)^\infty)_{1:n}}(x)
\Bigr).
\]
The conjugate words \(01p\) and \(10p\) generate periodic orbits with the same
cycle values and mean; hence the summand is an asymptotically periodic
zero-mean sequence plus an exponentially decaying transient. Abelian
convergence gives \(f_p\), and division by the resource limit gives
\(m_{\rm reg}=n_p f_p\).
\end{proof}

For \(x\in {\mathcal E}_{\rm K}\), Lemma~\ref{lem:kf_diag_resource} gives
\(g_{\beta,x}(x)=1-\beta\), hence \(m_\beta(x)=f_{\beta,x}(x)/(1-\beta)\). At a
Sturmian point \(x=y_s\), where \(0s\) is a Sturmian \({\mathcal M}\)-word, the required
normalized limit is
\begin{equation}
\label{eq:kf_sturmian_abelian_problem}
\lim_{\beta\uparrow1}
\frac{1}{1-\beta}
\sum_{n\ge1}\beta^{n-1}
\Bigl(
\phi_{(01s)_{1:n}}(y_s)
-
\phi_{(10s)_{1:n}}(y_s)
\Bigr).
\end{equation}
This is not a pure symbol-count functional; it involves nonlinear M\"obius
orbit values. The monotonicity of the discounted MP indices in
\cite[\S3.3.2]{danceSi19} suggests a squeeze proof through regular
Christoffel approximants.

\begin{conjecture}[\textup{(PCLI2)} completion for the Kalman-filter bandit]
\label{con:kf-pcli2}
The regular average MP index \(m_{\rm reg}\) admits a finite-valued,
nondecreasing, continuous extension \(m\colon[0,\infty)\to\mathbb R\), with
\(m(x)\to\infty\) as \(x\to\infty\). Moreover, \(m_\beta\to m\) locally
uniformly on compact subsets of \(\mathcal R_{\rm K}\), and
\(m_\beta(x)\to m(x)\) for every \(x\in {\mathcal E}_{\rm K}\). Equivalently,
\(f_{\beta,x}(x)/(1-\beta)\to m(x)\) on \({\mathcal E}_{\rm K}\). At a Sturmian point
\(y_s\), this value is the Abelian limit in
\eqref{eq:kf_sturmian_abelian_problem} and agrees with Christoffel
approximation limits from both sides. If
\[
        y_{01l^{(n)}}\uparrow y_s,
        \qquad
        y_{10u^{(n)}}\downarrow y_s
\]
are the Christoffel approximants in \cite[Lemma~55]{danceSi19}, then
\[
m(y_s)
=
\lim_{n\to\infty}m_{\rm reg}(y_{01l^{(n)}})
=
\lim_{n\to\infty}m_{\rm reg}(y_{10u^{(n)}}^-),
\]
where the last term is the left limit from the corresponding Christoffel
interval.
\end{conjecture}

\subsubsection{\textup{(PCLI3)}: threshold and marginal PCL identities}
\label{ss:kf-pcli3}

We finally consider the PCL identities. The regular metric propositions above
identify the gain, bias, marginal, and regular diagonal MP objects on
periodic-itinerary regions. On finite unions of such regular regions, the
discounted Radon--Nikodym PCL identity and the discounted marginal
integration-by-parts identity pass to the average limit as in the preceding
examples. The unresolved part is the singular symbolic-itinerary set, where
threshold atoms and exceptional diagonal ratios must be governed by the same
continuous extension \(m\).

\begin{conjecture}[\textup{(PCLI3)} completion for the Kalman-filter bandit]
\label{con:kf-pcli3}
The regular threshold gain functions \(F\) and \(G\), defined on the regular
threshold regions above, admit c\`adl\`ag locally bounded-variation extensions
to \(\mathbb R\), with \(G\) nonincreasing and nonconstant. For each \(x\), the limiting marginal metrics
\(f_x\) and \(g_x\) are finite-valued and c\`adl\`ag in the threshold variable.
With the extension \(m\) of Conjecture~\ref{con:kf-pcli2},
\[
        F(z_2)-F(z_1)
        =
        \int_{(z_1,z_2]}m(z)\,\mathrm dG(z),
        \qquad z_1<z_2,
\]
and
\[
f_x(z)-m(z)g_x(z)
=
\begin{cases}
\displaystyle
\int_{(z,x)}g_x(y)\,\mathrm dm(y), & z<x,\\[2.5ex]
\displaystyle
-\int_{[x,z]}g_x(y)\,\mathrm dm(y), & x\le z .
\end{cases}
\]
\end{conjecture}

\begin{remark}[Status of the Kalman-filter PCL conditions]
\label{re:kf_missing_piece}
The regular periodic-itinerary part is proved: gains are cycle means, the
limiting diagonal marginal resource is positive, and the discounted MP indices
converge locally uniformly. At exceptional states, however,
\(g_{\beta,x}(x)=1-\beta\), so the limiting diagonal ratio is a removable
\(0/0\) singularity. The remaining work is precisely
Conjectures~\ref{con:kf-pcli1}--\ref{con:kf-pcli3}: off-diagonal average
marginal-resource positivity, continuous extension of the regular MP index to
exceptional states, and the singular threshold and marginal PCL identities.
This appears to require Abelian summability combined with the
Christoffel/Sturmian and M\"obius-orbit analysis of \cite{danceSi19}.
\end{remark}

\begin{proposition}[Conditional average-indexability for the Kalman-filter bandit]
\label{pro:kf_avg_conditional}
Suppose that the long-run average \(\lambda\)-price problems satisfy the
optimal-action characterization required in Assumption~\ref{ass:excharcosdp}
for the limiting gain and bias metrics of the Kalman-filter model. Suppose
also that Conjectures~\ref{con:kf-pcli1}--\ref{con:kf-pcli3} hold. Then the
scalar Kalman-filter project is threshold-indexable under the long-run average
criterion. Its Whittle index is \(m_{\rm reg}\) on \(\mathcal R_{\rm K}\) and
the continuous extension \(m\) on \({\mathcal E}_{\rm K}\).
\end{proposition}

\begin{proof}
The first hypothesis gives the action-comparison and optimal-action structure
required by Theorem~\ref{the:pcliii}.
Conjecture~\ref{con:kf-pcli1} and Lemma~\ref{lem:kf_diag_resource} give
\textup{(PCLI1)} with exceptional set \({\mathcal E}_{\rm K}\);
Conjecture~\ref{con:kf-pcli2} gives \textup{(PCLI2)}; and
Conjecture~\ref{con:kf-pcli3} gives the regularity requirements and both
identities in \textup{(PCLI3)}. The conclusion follows from
Theorem~\ref{the:pcliii}.
\end{proof}

Thus the Kalman-filter example transfers cleanly on regular periodic-itinerary
regions and isolates the exceptional symbolic-itinerary limits as the remaining
average-cost indexability task.

\section{Conclusions}
\label{s:concl}

This paper has developed a performance-metric PCL framework for proving
threshold-indexability of binary-action restless bandit projects on real
interval state spaces. The main theorem verifies threshold optimality and
Whittle indexability simultaneously: under marginal-resource positivity and the
marginal integration-by-parts identity, threshold-indexability is equivalent to
the MP-index monotonicity/continuity condition. When these conditions hold, the
Whittle index is the MP index and the optimal threshold maps are its
generalized inverses. The formulation also covers exceptional states, where the
diagonal MP ratio is undefined and the index is supplied by continuous
extension or by a vanishing-discount limit.

For the discrete-time long-run average criterion, the paper gives a
vanishing-discount transfer theorem. Discounted threshold metrics are separated
into gain and bias terms, and marginal metrics, MP indices, and PCL identities
are passed to the limit. A key point is that the limiting marginal identity is
written with integrator \(\mathrm dm\), not \(\mathrm dg_x\). This distinction
is essential in average-reward real-state models where the limiting marginal
resource, as a function of the threshold, need not have locally bounded
variation.

The web-crawling and noisy-channel examples give complete long-run average
threshold-indexability results and recover known Whittle indices through a
common metric-transfer route. For the scalar Kalman-filter bandit, the paper
proves the regular periodic-itinerary part of the average transfer and
identifies the exceptional-state completions required for a full theorem. The
remaining indexability problem is reduced to three explicit metric-limit tasks
corresponding to \textup{(PCLI1)}--\textup{(PCLI3)}: off-diagonal average
marginal-resource positivity, continuous extension of the regular MP index to
exceptional symbolic-itinerary states, and the singular threshold and marginal
PCL identities associated with that extension.

The most immediate continuation is to prove these Kalman-filter exceptional
limits and to establish the average-optimality prerequisites for the associated
noncompact, unbounded-cost price problems. Broader directions include
continuous-time projects, bias-optimality criteria, and multidimensional state
spaces such as covariance-matrix sensor-scheduling models. In those settings,
the performance-metric viewpoint remains promising, but new notions of
structured action regions and marginal PCL identities will be needed.

\section*{Funding}
This work was supported in part by Universidad Carlos III de Madrid through
an internal research program grant and a grant for the acquisition of research
tools. 

\section*{Declaration of competing interest}
The author declares that there are no known competing financial interests
or personal relationships that could have appeared to influence the work
reported in this paper.

\section*{Declaration of generative AI and AI-assisted technologies in the manuscript preparation process}

During the preparation of this work, the author used OpenAI's ChatGPT
to assist with checking the manuscript for consistency, coherence,
clarity, presentation issues, and possible LaTeX or typographical
problems. The tool was not used to generate original mathematical
results, proofs, data, figures, or references. After using this tool,
the author reviewed and edited the content as needed and takes full
responsibility for the content of the publication.

\end{document}